\documentclass[12pt,oneside,english]{amsart}

\usepackage{tcolorbox}
\usepackage{graphicx}
\usepackage{caption}

\usepackage{amsmath}
\usepackage{amssymb}
\usepackage{hyperref}
\usepackage{enumerate}
\usepackage{graphicx}
\usepackage{color}
\usepackage{xcolor}
\usepackage{mathtools}
\usepackage{float}
\usepackage{tikz}
\usetikzlibrary{arrows}
\usepackage{bm}
\usepackage[title]{appendix}
\usepackage{mathrsfs}
\DeclareMathAlphabet\euscript{U}{eus}{m}{n}
\usepackage{hyperref}
\hypersetup{
	colorlinks = true, 
	urlcolor = blue, 
	linkcolor = red, 
	citecolor = blue 
}

\usepackage{mathdots}
\usepackage{amsthm}
\usepackage{tikz,amsmath}
\usetikzlibrary{arrows}
\usetikzlibrary{calc}
\usepackage{tikz-cd}
\usetikzlibrary{decorations.pathreplacing}
\tikzset{%
	show curve controls/.style={
		postaction={
			decoration={
				show path construction,
				curveto code={
					\draw [blue,-] 
					(\tikzinputsegmentfirst) -- (\tikzinputsegmentsupporta)
					(\tikzinputsegmentlast) -- (\tikzinputsegmentsupportb);
					\fill [red, opacity=0.5] 
					(\tikzinputsegmentsupporta) circle [radius=.2ex]
					(\tikzinputsegmentsupportb) circle [radius=.2ex];
				}
			},
			decorate
		}
	},
	scc/.style={
	},
}
\usepackage{epsfig}
\usepackage[normalem]{ulem}
\usepackage[all,line,
]{xy} 
\CompileMatrices

\newcommand{\id}{\operatorname{id}}

\newcommand{\Per}{\operatorname{Per}}

\newcommand{\sync}{\operatorname{sync}}

\newcommand{\K}{\operatorname{\underline{\mathbb K}}}

\newtheorem{thm}{Theorem}[section]
\newtheorem{corollary}[thm]{Corollary}
\newenvironment{cor}{\begin{corollary}\rm}{\end{corollary}}
\newtheorem{lemma}[thm]{Lemma}

\newtheorem{prop}[thm]{Proposition}

\newtheorem{assumption}[thm]{Assumption}

\newtheorem{definition}[thm]{Definition}

\newtheorem{example}[thm]{Example}

\newtheorem{question}[thm]{Question}

\newtheorem{remark}[thm]{Remark}
\newenvironment{rem}{\begin{remark}\rm}{\end{remark}}
\newtheorem{algorithm}[thm]{Algorithm}

\newtheorem{observation}[thm]{Observation}

\newtheorem{claim}[thm]{Claim}

\newtheorem{fact}[thm]{Fact}

\newenvironment{ex}{\begin{example}\rm}{\end{example}}
\newenvironment{re}{\begin{remark}\rm}{\end{remark}}

\newcommand{\vertiii}[1]{{\left\vert\kern-0.25ex\left\vert\kern-0.25ex\left\vert #1
    \right\vert\kern-0.25ex\right\vert\kern-0.25ex\right\vert}}

\title{The extended future cover and desynchronization of sofic subshifts}

\author{Klaus Thomsen}

\address{Department of Mathematics, Aarhus University, Ny Munkegade, 8000 Aarhus C, Denmark}

\email{matkt@math.au.dk}

\subjclass[2020]{Primary ; Secondary }

\date{\today}

\begin{document}

\begin{abstract} The paper constructs a cover of the future cover of a sofic subshift $Y$ which is canonical in the same way as the future cover itself. In some cases the cover is isomorphic to the future cover and in other it is a genuine extension. The extended future cover is then used to show that by removing from a sofic subshift all elements containing a synchronizing word, the result is a new sofic subshift whose future cover  can be realized as the merged graph of a canonical subgraph of the extended future cover of $Y$.
\end{abstract}
\maketitle


\section{Introduction}

The purpose of this paper is to exhibit a natural extension of the future cover of a sofic subshift and apply it to the study of its structure. The future cover was introduced by Wolfgang Krieger in \cite{Kr1} and \cite{Kr2} and has the remarkable property that conjugacies of sofic subshifts lift to unique conjugacies of the associated covers. His methods and results were extended in a recent work by the author, \cite{Th1}, giving rise to other canonical covers with the same nice property, but Krieger's work and his future cover remains the source of all known examples of this sort. It is therefore useful to have a way to find the future cover of a sofic subshift starting from an arbitrary presentation of it. This is the first goal here, and is really not that hard to achieve: An appropriate subgraph of what is termed the subset construction on page 76 of \cite{LM} turns out to factor onto the future cover via what is called merging, or the merged graph in \cite{LM}. A very similar, but not quite identical way of obtaining the future cover from the subset construction was described by Nasu in \cite{N}.  

The more difficult task is to show that when applied to an appropriate presentation of the given sofic subshift, this subgraph of the subset construction is itself a cover which is strongly canonical in the same sense as the future cover is. The proof of this is based on a further development of the methods and results from \cite{Th1} and in the process we prove that the subset construction applied to a weakly canonical cover is itself a weakly canonical cover. Both the new strongly canonical cover from \cite{Th1} and the one we obtain here depend on the canonicity of the future cover which was the main result of \cite{Kr1} and \cite{Kr2}. It is therefore appropriate to consider the paper \cite{Th1} as a first addendum to Krieger's work and the present paper as a second addendum. The main difference between the covers from the two addenda is that the one from \cite{Th1} does not in general factor onto Krieger's future cover while the one we obtain here does so in a canonical way; a fact which may justify that we call it \emph{the extended future cover}.

As explained in the introduction to \cite{Th1} the motivation for the investigation of canonical covers of sofic subshifts is the desire to obtain a better understanding of their structure, and in particular the structure which set them apart from subshifts of finite type. We show in this paper one way the future cover and its extended version can be used for this purpose. Specifically, it is shown that by removing from a sofic subshift $Y$ the elements that contain a synchronizing word the result is a sofic subshift $\eth Y$ of $Y$ which we call the desynchronization of $Y$. When $Y$ is of finite type the desynchronization is empty and when $Y$ is non-wandering it agrees with the derived shift space $\partial Y$ which was introduced in \cite{Th0}. For general sofic subshifts the two notions disagree and as a result the sofic subshifts $\eth^n Y$ and $\partial^n Y$ that arise by iterating the constructions differ wildly, also for irreducible and mixing sofic subshifts. Since all the successive desynchronizations $\eth^nY$ are conjugacy invariants of $Y$, their introduction raises many questions which deserve attention, but in this paper we show only how the future cover of $\eth Y$ can be obtained as the merged graph of a canonical subgraph of the extended future cover of $Y$.


\section{Terminology and notation}

Let $A$ be a finite set. The set $A^\mathbb Z$ of bi-infinite sequences 
$$
x = (x_i)_{i \in \mathbb Z} = \cdots x_{-3}x_{-2}x_{-1}x_0x_1x_2x_3 \cdots
$$
of elements from $A$ is a compact metric space. The shift $\sigma$ is the homeomorphism of $A^\mathbb Z$ defined such that
$$
\sigma(x)_i := x_{i+1}.
$$
A closed subset $X \subseteq A^\mathbb Z$ is shift-invariant when $\sigma(X) = X$ and it is then called a \emph{subshift}. We shall mainly consider \emph{subshifts of finite type} (abbreviated to SFT) and \emph{sofic subshifts}. See Definition 2.1.1 and Definition 3.1.3 in \cite{LM}.

The set $\mathbb W(X)$ of \emph{words} in a subshift $X$ consists of the finite strings $a_1a_2 \cdots a_n \in A^n$ with the property that $a_1a_2 \cdots a_n  = x_1x_2\cdots x_n$ for some element $x = (x_i)_{i \in \mathbb Z} \in X$. Given an element $x \in X$ and an integer $j \in \mathbb Z$ we denote by $x_{(-\infty,j]}$ the element 
$$
\cdots x_{k}x_{k+1} x_{k+2} \cdots x_j \in A^{(-\infty,j]},
$$
and by $X(-\infty,j]$ the set
$$
X(-\infty,j] := \left\{ x_{(-\infty,j]}: \ x \in X \right\}.
$$
Symbols like $x_{[j,\infty)}, \ x_{[i,j]},   X[j,\infty)$ and $X[i,j]$ have a similar meaning. The set of points in a subshift $X$ that are periodic under the shift will be denoted by $\Per(X)$.

 Let $G$ be a finite directed graph with arrows or edges $E_G$ and vertices $V_G$. For $e \in E_G$ let $s_G(e) \in V_G$ be the start vertex and $t_G(e)$ the terminal vertex of $e$. We extend the definition of $t_G$ to finite and left-infinite paths in $G$, and the definition of $s_G$ to finite and right-infinite paths in $G$, in the obvious way. We denote the edge shift of $G$ by $X_G$. Thus $X_G$ consists of the bi-infinite sequences $(e_i)_{i \in \mathbb Z}$
 of edges in $G$ such that $s_G(e_{i+1}) = t_G(e_i)$ for all $i \in \mathbb Z$, and we shall sometimes refer to the elements of $X_G$ as \emph{rays} in $G$. The words in $\mathbb W(X_G)$ are the finite paths in $G$ that sit as a subpath of a ray in $G$. The set $X_G$ is an SFT and every SFT is conjugate to $X_G$ for some graph $G$ by Theorem 2.3.2 in \cite{LM}.

In this paper the components in $G$ and $X_G$ will be important. Recall that a directed graph $G$ is strongly connected when every pair of vertices $v,w \in V_G$ can be connected by a finite path in $G$; that is, there is a finite path $\gamma$ in $G$ such that $s_G(\gamma) = v$ and $t_G(\gamma) = w$. This happens if and only if the shift acts irreducibly on $X_G$, as defined in \cite[Definition 1.3.6]{LM}. In a general finite directed graph $G$, the maximal strongly connected subgraphs $C$ of $G$ are the \emph{components} of $G$ and the corresponding subshifts $X_C \subseteq X_G$ are the components of $X_G$. Then the disjoint union
$$
\bigsqcup_C X_C = \overline{\Per (X_G)}
$$
consists of the non-wandering elements of $X_G$.

Let $A$ be a finite set (the alphabet) and $L_G : E_G \to A$ a map which we consider as a labeling of the edges in $G$. Then $(G,L_G)$ is a \emph{labeled graph}, $Y:= L_G(X_G) \subseteq A^\mathbb Z$ is a sofic subshift and $(G,L_G)$ is a \emph{presentation} of $Y$, cf. \S 3.1 in \cite{LM}. We will always assume, as we can without loss of generality, that the graph $G$ of a presentation $(G,L_G)$ does not contain sources or sinks. 
Given a finite path $\gamma:=x_{[i,j]} \in X_G[i,j]$ in $G$ we define $L_G(\gamma) \in Y[i,j]$ such that $L_G(\gamma)_k := L_G(x_k)$ for $k \in [i,j]$. The extension to left- and right-infinite paths is made in a similar way. In particular, when $x=(x_i)_{i \in \mathbb Z} \in X_G$,
\begin{align*}
&L_G(x) := \left(L_G(x_i)\right)_{i \in \mathbb Z} = \cdots L_G(x_{-3}) L_G(x_{-2}) L_G(x_{-1}) L_G(x_0)L_G(x_1)L_G(x_2) L_G(x_3) \cdots \\
&\in Y .
\end{align*}

 
  When $D \subseteq A^\mathbb N$ is a subset of $A^{\mathbb N}$ and $a \in A$, we let $\frac{D}{a}$ denote the set
$$
\frac{D}{a} := \left\{x_1x_2x_3 \cdots \in A^\mathbb N : \ ax_1x_2\cdots \in D \right\} = \left\{x \in A^\mathbb N: \ ax \in D\right\};
$$
i.e. $\frac{D}{a}$ is the set of elements of $A^{\mathbb N}$ that are obtained from the elements of $D$ that start with $a$ by deleting the first coordinate.

Let $Y \subseteq A^\mathbb Z$ be a sofic subshift. The \emph{future set} $F(y)$ of an element $y \in Y$ is 
$$
F(y) := \left\{w \in Y[0,\infty): \ y_{(-\infty,-1]}w \in Y\right\}.
$$
The \emph{future cover} $(\mathbb K(Y),L_{\mathbb K(Y)})$ of $Y$ is the labeled graph where
$$
V_{\mathbb K(Y)} := \left\{F(y): \ y \in Y\right\}
$$
and there is a labeled arrow $F(y) \overset{a}{\to} F(y')$ in $\mathbb K(Y)$ when 
$$
F(y') = \frac{F(y)}{a} .
$$

The notion of regularity of a labeled graph played a significant role in \cite{Th1}, partly as a tool to characterize the future cover, and we shall need it here again. A labeled graph $(H,L_H)$ is \emph{regular} when every vertex $v \in V_H$ has the property that there is an element $z \in X_H$ such that $t_H(z_{(-\infty,-1]}) = v$ and 
$$
\left\{L_H(x): \ x \in X_H[0,\infty), \ s_H(x) = v \right\} =F(L_H(z)).
$$
We shall need at least two more properties that a labeled graph can have; it can be \emph{right-resolving}, as in Definition 3.3.1 of \cite{LM}, and it can be \emph{follower-separated} as in Definition 3.3.7 of \cite{LM}. For other notions related to subshifts that are not explained we refer also to \cite{LM}.

\section{The subset construction and the future cover}

\subsection{The subset contruction}

Let $(G,L_G)$ be a presentation of the sofic shift $Y$. When $D\subseteq V_{G}, \ D \neq \emptyset$, and $a \in A$, set
$$
[D,a]: = \left\{ t_G(e): \ e \in E_G, \ s_G(e) \in D, \ L_G(e) = a \right\};
$$
the set of vertices that can be reached from an element of $D$ using an edge in $G$ labeled by $a$. We consider the collection $2^{V_G}$ of non-empty subsets of $V_G$ as the set of vertices in a graph $\overline{G}$ with edges 
$$
E_{\overline{G}} := \left\{(D,a) \in 2^{V_{G}} \times A : \ [D,a] \neq \emptyset \right\} .
$$
The start vertex $s_{\overline{G}}(D,a)$ of $(D,a) \in E_{\overline{G}}$ is $D$ and the terminal vertex $t_{\overline{G}}(D,a)$ is $[D,a]$. We define a labeling $L_{\overline{G}}$ of $\overline{G}$ such that
$$
L_{\overline{G}}(D,a) := a 
$$
when $(D,a) \in E_{\overline{G}}$. Then $(\overline{G},L_{\overline{G}})$ is a right-resolving labeled graph and 
$L_{\overline{G}}(X_{\overline{G}}) =  Y$, cf. Theorem 3.3.2 in \cite{LM}. It is called the subset construction in \cite{LM}. Since $(G,L_G)$ is a presentation of $Y$, $\overline{G}$ will not contain sinks, but it may contain sources even though $G$ doesn't. Hence, in general $(\overline{G},L_{\overline{G}})$ first becomes a presentation of $Y$ in the strict sense in which we use the term, after sinks have (succesively) been removed.

\subsection{Merging}
The subset construction is a well-known construction in automata theory and so is the following construction which is referred to as minimization in \cite{N}. The focus here is on symbolic dynamics rather than on formal languages as in the theory of automata, and since the settings are also slightly different, we will use the terminology suggested in \cite{LM} and call the construction \emph{merging}.

Let $(H,L_H)$ be a labeled graph and assume that $H$ does not contain sinks. Given a vertex $v \in V_H$ the follower set $f_H(v)$ is the set
$$
f_H(v) := \left\{L_H(x): \ x \in X_H[0,\infty), \ s_H(x) = v \right\}.
$$ 
The \emph{merged graph} $([H],L_{[H]})$ of $(H,L_H)$ is then defined as follows, cf. page 78 of \cite{LM}. The set of vertices $V_{[H]}$ is the set
$$
V_{[H]} := \left\{f_H(v): \ v \in V_H\right\}
$$
and for $a \in A$ there is a labeled arrow $f_H(v) \overset{a}{\to} f_H(w)$ in $[H]$ when there are vertices $v',w' \in V_H$ such that $f_H(v') = f_H(v), \ f_H(w') = f_H(w)$ and there is a labeled arrow $v' \overset{a}{\to} w'$ in $H$. When $(H,L_H)$ is right-resolving we can define a map $f_H: E_H \to E_{[H]}$ such that $f_H(e) \in E_{[H]}$ is the arrow with the properties
\begin{itemize}
\item $s_{[H]}(f_H(e)) = f_H\left(s_H(e)\right)$, and
\item $L_{[H]}(e) = L_H(e)$.
\end{itemize}
Then $f_H:H \to [H]$ is a labeled-graph homomorphism.

 \begin{lemma}\label{03-11-24x} Assume that $(H,L_H)$ is right-resolving. Let $v,w \in V_{H}$. There is a labeled arrow $f_H(v) \overset{a}{\to} f_H(w)$ in $([H],L_{[H]})$ if and only if there is a labeled arrow $v \overset{a}{\to} u$ in $(H,L_H)$ such that $f_H(u) =f_H(w)$. 
\end{lemma} 
\begin{proof} Assume that $f_H(v) \overset{a}{\to} f_H(w)$ in $[H]$ and let $v_1, w_1 \in V_{H}$ be such that $v_1 \overset{a}{\to} w_1$ in $H$, $f_H(v_1) = f_H(v)$ and $f_H(w_1) = f_H(w)$. Since $f_H(v) = f_H(v_1)$ there is an arrow $e$ in $E_{H}$ labeled $a$ which leaves $v$. Set $u:= t_{H}(e)$. By using that $(H,L_{H})$ is right-resolving we find that
$$
f_H(u) = \frac{f_H(v)}{a} = \frac{f_H(v_1)}{a} = f_H(w_1) .
$$
This establishes the 'only if' part. The 'if' part is trivial.
\end{proof}

\begin{lemma}\label{10-09-25} Assume that $(H,L_H)$ is right-resolving. Then $([H],L_{[H]})$ is right-resolving and follower-separated, and there is a labeled-graph homomorphism $f_H: H \to [H]$ such that $f_H: X_H \to X_{[H]}$ is a factor map. If $(H,L_H)$ is regular, then so is $([H],L_{[H]})$.
\end{lemma}
\begin{proof} $([H],L_{[H]})$ is right-resolving and follower-separated by Lemma 3.3.8 of \cite{LM}. To see that $f_H: X_H \to X_{[H]}$ is surjective, consider a finite path 
\begin{equation}\label{03-09-25d}
f_H(v_1) \overset{a_1}{\to} f_H(v_2) \overset{a_2}{\to} f_H(v_3) \overset{a_3}{\to} \cdots \cdots \overset{a_{n-1}}{\to} f_H(v_n)
\end{equation}
in $[H]$. Repeated applications of Lemma \ref{03-11-24x} give a path 
$$
v_1 \overset{a_1}{\to} u_2 \overset{a_2}{\to}  u_3 \overset{a_3}{\to} \cdots \cdots \overset{a_{n-1}}{\to} u_n
$$ 
in $H$ which is mapped to \eqref{03-09-25d} under $f_H$. The surjectivity of $f_H$ follows therefore from compactness of $ X_{H}$. 

Assume now that $(H,L_H)$ is regular and consider a vertex $v \in V_H$. There is a ray $z' \in X_H$ such that $t_H(z'_{(-\infty,-1]}) = v$ and $f_H(v) = F(L_H(z'))$. Set $z:= f_H(z') \in X_{[H]}$ and note that $t_{[H]}(z_{(-\infty,-1]}) = f_H( t_H(z'_{(-\infty,-1]})) = f_H(v)$. An application of Lemma \ref{03-11-24x} as above shows that $f_{[H]}(f_H(v)) = f_H(v)$. Since $L_{[H]}(z) = L_H(z')$ we conclude first that $f_{[H]}(f_H(v)) = F(L_{[H]}(z))$ and then that $([H],L_{[H]})$ is regular.
\end{proof}

We note that in the setting of Lemma \ref{10-09-25} the labeled-graph homomorphism $f_H: H \to [H]$ is injective on $V_H$ if and only if it is an isomorphism of labeled graphs, and this happens if and only if $(H,L_H)$ is follower-separated.

\begin{rem}\label{24-07-266}
As mentioned above, the construction of the merged graph is well-known. It was invented in the theory of formal languages to produce a right-resolving presentation out of one which may not have this property, and this may be the reason that it apparently has not been noticed to produce a factor map when applied to a right-resolving presentation. This property is the main reason it appears in this paper.
\end{rem}

\subsection{From the subset construction to the future cover}\label{consec}

Let $(G,L_G)$ be a presentation of the sofic shift $Y =L_G(X_G) \subseteq A^\mathbb Z$. When $y \in Y$, set
$$
D^y:= \left\{ t_G(x) : \ x \in X_G(-\infty,-1], \ L_G(x) = y_{(-\infty,-1]} \right\} .
$$
The set of vertices $V_{\underline{G}} := \left\{D^y: \ y \in Y\right\}$ is a hereditary subset of vertices in $V_{\overline{G}}$, and when we set $E_{\underline{G}} := \left\{e \in E_{\overline{G}}: \ s_{\overline{G}}(e) \in V_{\underline{G}}\right\}$, we have defined a graph $\underline{G}$, and hence also a labeled graph $(\underline{G}, L_{\underline{G}})$ where $L_{\underline{G}}:= L_{\overline{G}}|_{E_{\underline{G}}}$. Thus $(\underline{G},L_{\underline{G}})$ is a hereditary labeled subgraph of $(\overline{G},L_{\overline{G}})$ in the sense of \cite{Th1}.

Note that we can define a shift-commuting map $\alpha_G : Y \to X_{\underline{G}}$ such that 
\begin{itemize}
\item[a)] $s_{\underline{G}}(\alpha_G(y)_k) = D^{\sigma^k(y)}$ for all $k \in \mathbb Z$, and
\item[b)] $L_{\underline{G}}\left(\alpha_G(y)\right) = y$.
\end{itemize}
In particular, $L_{\underline{G}}(X_{\underline{G}}) = Y$ and $(\underline{G}, L_{\underline{G}})$ is a right-resolving presentation of $Y$.

\begin{lemma}\label{07-02-26} $(\underline{G},L_{\underline{G}})$ is regular.
\end{lemma}
\begin{proof} Let $y \in Y$. To show that $D^y \in V_{\underline{G}}$ is regular, note that $t_{\underline{G}}(\alpha_G(y)_{(-\infty,-1]}) = D^y$. Let $z \in Y[0,\infty)$ such that $L_{\underline{G}}(\alpha_G(y)_{(-\infty,-1]})z = y_{(-\infty,-1]}z \in Y$. There is an $x \in X_G$ such that $L_G(x) =y_{(-\infty,-1]}z$. Then $t_G(x_{(-\infty,-1]}) \in D^y$ and it follows therefore from the definition of $(\underline{G},L_{\underline{G}})$ that $z \in f_{\underline{G}}(D^y)$.
\end{proof}


\begin{prop}\label{07-02-26a} There is an isomorphism $\theta: \ ([\underline{G}],L_{[\underline{G}]})     \to (\mathbb K(Y),L_{\mathbb K(Y)})$ of labeled graphs such that
$$
\theta(f_{\underline{G}}(D^y)) = F(y)
$$
for all $y \in Y$.
\end{prop}
\begin{proof} It follows from Lemma \ref{10-09-25} and Lemma \ref{07-02-26} that $([\underline{G}],L_{[\underline{G}]})$ is right-resolving, follower-separated and regular. It follows therefore from Proposition 3.1 of \cite{Th1} and its proof that there is an injective labeled-graph homomorphism $\theta: \ ([\underline{G}],L_{[\underline{G}]}) \to (\mathbb K(Y),L_{\mathbb K(Y)})$ such that 
$$
\theta(f_{\underline{G}}(D^y)) = F(L_{[\underline{G}]}(f_{\underline{G}}(\alpha_G(y)))).
$$
Since $L_{[\underline{G}]}(f_{\underline{G}}(\alpha_G(y))) = L_{\underline{G}}(\alpha_G(y)) = y$ we see that $\theta$ maps $V_{[\underline{G}]}$ onto $V_{\mathbb K(Y)}$. The range of $\theta$ is a hereditary subgraph of $\mathbb K(Y)$ by Proposition 3.1 in \cite{Th1} so it follows that $\theta$ is also surjective.
\end{proof}

 The last proposition allows us to make the following identification which we shall use from now on:

\begin{corollary}\label{04-09-25d}  $(\mathbb K(Y),L_{\mathbb K(Y)}) = ([\underline{G}],L_{[\underline{G}]})$ as labeled graphs.
\end{corollary}

Combined with Lemma \ref{10-09-25} we get also the following 

\begin{corollary}\label{24-07-26b} The merging $f_{\underline{G}} : X_{\underline{G}} \to X_{\mathbb K(Y)}$ is a factor map.
\end{corollary}

\begin{rem}\label{24-07-26a} As mentioned in the introduction, Nasu gave in \cite{N} a description of another way of using the subset construction to get from an arbitrary presentation of a sofic shift to its future cover, or Krieger cover as Nasu calls it. He does this by first realizing the follower set graph, whose definition can be found on page 73 of \cite{LM}, as the merged graph of the subset construction and then he identifies the future cover with a labeled subgraph of the follower set graph. As in this paper the procedure involves merging and the passage to a subgraph but in different order: While Nasu first merges and then passes to a subgraph, we do it here in the opposite order. As a result the future cover appears here as a factor of an edge-shift rather than a subshift of one.  
\end{rem}

The right-inverse $\alpha_G$ for the labeling $L_G$ which we introduced above, and the right-inverse $\alpha_{Y}$ for the of labeling $L_{\mathbb K(Y)}$ of the future cover which was introduced by Krieger in \cite{Kr1} and used in \cite{Th1}, will play  important roles in the following. Recall that $\alpha_Y : Y \to X_{\mathbb K(Y)}$ is defined such that  
\begin{itemize}
\item[a)] $s_{\mathbb K(Y)}(\alpha_Y(y)_i) = F(\sigma^i(y)), \ i \in \mathbb Z$, and
\item[b)] $L_{\mathbb K(Y)}(\alpha_Y(y)) = y$.
\end{itemize}
The two right-inverses for the labelings are neatly related via the merging:

\begin{lemma}\label{27-01-26b} $f_{\underline{G}} \circ \alpha_G = \alpha_Y$.
\end{lemma}
\begin{proof} Let $y \in Y$. Then $L_{\mathbb K(Y)}( f_{\underline{G}} \circ \alpha_G(y)) = L_{\underline{G}}(\alpha_G(y)) = y = L_{\mathbb K(Y)}(\alpha_Y(y))$, and 
$$
s_{\mathbb K(Y)} (f_{\underline{G}} \circ \alpha_G(y)_k) = f_{\underline{G}}(s_{\underline{G}}(\alpha_G(y)_k)) =
f_{\underline{G}}(D^{\sigma^k(y)}) = F(\sigma^k(y)) = s_{\mathbb K(Y)}(\alpha_Y(y)_k)
$$
for all $k \in \mathbb Z$. Since $(\mathbb K(Y),L_{\mathbb K(Y)})$ is right-resolving it follows that $f_{\underline{G}} \circ \alpha_G(y) = \alpha_Y(y)$.
\end{proof}

\section{The canonicity of the subset construction}

The labeled graph  $(\overline{G},L_{\overline{G}})$ contains an important subgraph $(G'',L_{G''})$ which was introduced and used in \cite{Th1}, and it will play a key role also in this paper. The set $V_{G''}$ of vertices consists of the non-empty subsets $F$ of vertices from $V_G$; i.e. $V_{G''} = V_{\overline{G}}$. For $F,F' \in  V_{G''}$ there is an edge $e' \in E_{G''}$ with $s_{G''} (e') = F$, $t_{G''}(e') = F'$ and label $L_{G''}(e') = a$ when
\begin{equation*}\label{23-06-25ay}
F \subseteq \left\{ s_G(e): \ e \in E_G, \ L_G(e) = a \right\} 
\end{equation*}
and
$$
F' = \left\{ t_G(e) : \ e \in E_G, \ s_G(e) \in F, \ L_G(e) = a \right\} .
$$ 
Thus $(G'',L_{G''})$ is the labeled subgraph of $(\overline{G},L_{\overline{G}})$ obtained by only allowing the arrows from $\overline{G}$ where all elements from $V_G$ in the start vertex emit an arrow with the given label. The other parts of $(\overline{G},L_{\overline{G}})$, the vertices and the labeling, remain the same. In particular, $X_{G''} \subseteq X_{\overline{G}}$. It follows from Lemma 4.2 in \cite{Th1} that $L_{G''}(X_{G''}) =Y$, but in general, $(G'',L_{G''})$ may contain both sources and sinks and hence only becomes a presentation of $Y$ after these have (succesively) been removed.

The following is Corollary 4.3 in \cite{Th1}.

\begin{lemma}\label{12-02-26bx}  Let $(G,L_G)$ and $(H,L_H)$ be right-resolving presentations of the sofic subshifts $Y$ and $Z$, respectively. Assume that $\psi : Y \to Z$ and $\phi : X_G \to X_H$ are conjugacies such that
\begin{equation}\label{09-10-25axy}
\xymatrix{
 X_G\ar[r]^-{\phi} \ar[d]_{L_G}& X_{H} \ar[d]^-{L_H}\\
   Y \ar[r]^{\psi} & Z}
\end{equation}
commutes. There is a conjugacy $\phi'' : X_{G''} \to X_{H''}$ such that
\begin{equation}\label{17-11-25cx}
\xymatrix{
 X_{G''} \ar[r]^-{\phi''} \ar[d]_{L_{G''}}&  X_{H''} \ar[d]^{L_{H''}} \\
Y   \ar[r]^-\psi & Z}
\end{equation}
commutes.
\end{lemma}


 Note that an arrow $h$ in $G''$ consists of, or can be identified with, a finite set $N(h)$ of arrows from $G$, all with the same label, such that
$$
s_{G''}(h) = \left\{ s_G(e): \ e \in N(h)\right\}
$$
and
$$
t_{G''}(h)  = \left\{ t_G(e): \ e \in N(h)\right\}.
$$
When $(G,L_G)$ is right-resolving, an element $x \in X_{G''}$ determines a finite set $N(x) \subseteq X_G$ such that
$$
N(x_k) = \left\{ z_k: \ z \in N(x)\right\}, \ k \in \mathbb Z,
$$
and $L_{G''}(x_k) = L_G(z_k)$ for all $z \in N(x)$ and all $k \in \mathbb Z$.
In terms of this identification of elements in $X_{H''}$ with subsets of $X_H$, the conjugacy $\phi''$ of Lemma \ref{12-02-26bx} is given by the formula
$$
N(\phi''(x)) = \phi(N(x))
$$
for $x \in X_{H''}$. See \cite{Th1}.

Lemma \ref{12-02-26bx} is used to state and prove the following

\begin{thm}\label{25-03-26}  Let $(G,L_G)$ and $(H,L_H)$ be right-resolving presentations of the sofic subshifts $Y$ and $Z$, respectively, and assume that there are conjugacies $\phi: X_G \to X_H$ and $\psi : Y \to Z$ such that the diagram \eqref{09-10-25axy} in Lemma \ref{12-02-26bx} commutes.
There is a unique conjugacy 
$$
\overline{\phi} : X_{\overline{G}} \to X_{\overline{H}}
$$ 
with the following two properties: 
\begin{itemize}
\item[(1)] $\overline{\phi}$ extends $\phi''$, and
\item[(2)]\begin{equation*}\label{25-03-26a}
\xymatrix{
 X_{\overline{G}} \ar[r]^-{\overline{\phi}} \ar[d]_{L_{\overline{G}}}& X_{\overline{H}} \ar[d]^-{L_{\overline{H}}}\\
   Y \ar[r]_{\psi} & Z}
\end{equation*}
commutes. 
\end{itemize}
\end{thm}

In the remaining part of the section we assume we are in the setting of Theorem \ref{25-03-26}. We need some preparations for its proof.

\begin{lemma}\label{25-03-26b} Let $\gamma = (\gamma_k)_{k\in [1,n]}$ be a finite path in ${\overline{G}}$.  

\begin{itemize}
\item[(1)] With respect to inclusion the set
$$
B := \left\{ v \in s_{\overline{G}}(\gamma): \ v=s_G(\mu)  \ \ \text{for a finite path $\mu$ in $G$ with $L_G(\mu) = L_{\overline{G}}(\gamma)$} \right\}
$$
is maximal in 
$$
\left\{ M \subseteq  s_{\overline{G}}(\gamma): \ M = s_{G''}(\nu)  \ \text{for a finite path $\nu$ in $G''$ with $L_{G''}(\nu) = L_{\overline{G}}(\gamma)$} \right\}.
$$
\item[(2)] There is a unique path $\delta = (\delta_k)_{k \in [1,n]}$ in $G''$ with $s_{G''}(\delta) = B$ and $L_{G''}(\delta) = L_{\overline{G}}(\gamma)$. 
\item[(3)] This path, $\delta$, has the property that $t_{G''}(\delta) = t_{\overline{G}}(\gamma)$. 
\item[(4)] Set $m := \max \{ k \leq n: \ \#t_{\overline{G}}(\gamma_k) = \# t_{\overline{G}}(\gamma_n)\}$. Then $t_{G''}(\delta_j) = t_{\overline{G}}(\gamma_j)$ for $j \in [m,n]$.
\end{itemize}
\end{lemma}
\begin{proof} (1) + (2) + (3): By definition of $\overline{G}$, for any element $w$ of $t_{\overline{G}}(\gamma)$ there is a path $\mu$ in $G$ such that $s_G(\mu) \in s_{\overline{G}}(\gamma)$, $L_G(\mu) = L_{\overline{G}}(\gamma)$ and $t_G(\mu) = w$. Hence $B \neq \emptyset$, and by definition of $G''$ there is a path $\delta$ in $G''$ such that $s_{G''}(\delta) = B$ and $L_{G''}(\delta) = L_{\overline{G}}(\gamma)$. As such it is unique because $G''$ is right-resolving. The maximality of $B$ follows because $s_{G''}(\nu) \subseteq B$ for any path $\nu$ in $G''$ for which $s_{G''}(\nu) \subseteq s_{\overline{G}}(\gamma)$ and $L_{G''}(\nu) = L_{\overline{G}}(\gamma)$. The fact that $t_{G''}(\delta) \supseteq t_{\overline{G}}(\gamma)$ follows from the observation made in the first sentence of the proof. The reverse inclusion, $t_{G''}(\delta) \subseteq t_{\overline{G}}(\gamma)$, follows from the definition of $\overline{G}$ because $B \subseteq s_{\overline{G}}(\gamma)$.  

(4): Let $j \in [m,n-1]$. Since $t_{G''}(\delta) = t_{\overline{G}}(\gamma)$, for every $w \in t_{\overline{G}}(\gamma)$ there is a path $\nu$ in $G$ such that $s_G(\nu) \in t_{G''}(\delta_j)$, $L_G(\nu) = L_{\overline{G}}(\gamma)_{[j+1,n]}$ and $t_G(\nu) = w$. Since $(G,L_G)$ is right-resolving this implies that $\# t_{G''}(\delta_j) \geq \# t_{\overline{G}}(\gamma)$. Since $t_{G''}(\delta_i)  \subseteq t_{\overline{G}}(\gamma_i)$ for all $i$, and $\# t_{\overline{G}}(\gamma_i) = \# t_{\overline{G}}(\gamma)$ when $i \in [m,n]$, we have that $\# t_{G''}(\delta_j) \leq \# t_{\overline{G}}(\gamma)$, and hence also that $\# t_{G''}(\delta_j) = \# t_{\overline{G}}(\gamma)$ and $t_{G''}(\delta_j)  = t_{\overline{G}}(\gamma_j)$ for all $j \in [m,n]$.

 
\end{proof}
The path $\delta$ in Lemma \ref{25-03-26b} will be referred to as \emph{the maximal $G''$-path dominated by $\gamma$.}

Given a path $\eta = (\eta_k)_{k \in [i,j]}$ in $\overline{G}$ and an interval $I \subseteq [i,j]$ we say that the restriction  $\eta_I$ of $\eta$ to $I$ is \emph{homogeneous} in $\eta$ when $\# s_{\overline{G}}(\eta_k) = \# t_{\overline{G}}(\eta_k)$ for all $k \in I$. When $C$ is a component in $\overline{G}$  the number $\# F$ is the same for every vertex $F \in V_C$ and we denote this common value by $M(C)$. In particular, any path $\eta =(\eta_k)_{k \in [i,j]}$ in $C$ will be homogeneous and $\# s_{\overline{G}}(\eta_k)=\# t_{\overline{G}}(\eta_k) = M(C)$ for all $k \in [i,j]$. 

\begin{lemma}\label{13-02-26cx} Let $\gamma = (\gamma_k)_{k\in [1,n]}$ be a finite path in ${\overline{G}}$ and let $\delta$ be the maximal $G''$-path dominated by $\gamma$. Assume that $\gamma_{[1,a-1]}$ and $\gamma_{[b,n]}$ are homogeneous in $\gamma$ and that $1<a < b<n$. Then $\delta_{[a,b-1]}$ is the maximal $G''$-path dominated by $\gamma_{[a,b-1]}$.   
\end{lemma}
\begin{proof} Let $\delta'$ denote the maximal $G''$-path dominated by $\gamma_{[a,b-1]}$. Let
$$
B := \left\{ v \in s_{\overline{G}}(\gamma): \ v=s_G(\mu)  \ \ \text{for a finite path $\mu$ in $G$ with $L_G(\mu) = L_{\overline{G}}(\gamma)$} \right\}
$$
and
$$
B' := \left\{ v \in s_{\overline{G}}(\gamma_{[a,b-1]}): \ v=s_G(\mu)  \ \ \text{for a finite path $\mu$ in $G$ with $L_G(\mu) = L_{\overline{G}}(\gamma_{[a,b-1]})$} \right\}.
$$
Since $\gamma_{[1,a-1]}$ is homogeneous there is a path $\nu$ in $G''$ with $s_{G''}(\nu) = B$, $t_{G''}(\nu) = B'$ and $L_{G''}(\nu) = L_{\overline{G}}(\gamma_{[1,a-1]})$. In addition, it follows from (3) and (4) of Lemma \ref{25-03-26b} that $\nu':= \gamma_{[b,n]}$ is a path in $G''$ with $s_{G''}(\nu') = t_{G''}(\delta') = t_{\overline{G}}(\gamma_{[1,b-1]})$, $t_{G''}(\nu') = t_{\overline{G}}(\gamma)$ and $L_{G''}(\nu') = L_{\overline{G}}(\gamma_{[b,n]})$. It follows then from (2) of Lemma \ref{25-03-26b} that $\nu \delta' \nu'$ is the maximal $G''$-path dominated by $\gamma$, i.e. $\nu \delta' \nu' = \delta$. Hence $\delta' = \delta_{[a,b-1]}$. 
\end{proof}
 Recall that when $\lambda : X \to X' \subseteq A^{\mathbb Z}$ is a continuous shift-commuting map between subshifts $X$ and $X'$, there is an $n \in \mathbb N$ and a map $\mathbb W_{2n+1}(X) \to A$, which we also denote by $\lambda$, such that
$$
\lambda(x)_i = \lambda(x_{[i-n,i+n]})
$$
for all $i \in \mathbb Z$. This is why such a map $\lambda$ is called a sliding block code. The number $n$ will here be called a \emph{window size} of $\lambda$. Note that any $m \geq n$ is then also a window size of $\lambda$. When a window size $n$ has been fixed and $w= (w_k)_{k \in [i,j]} \in \mathbb W_{j-i+1}(X)$ for some $j-i> 2n$, we define $\lambda(w) := (\lambda(w)_k)_{k  \in [i+n,j-n]}$ such that
$$
\lambda(w)_k = \lambda(w_{[k-n,k+n]})
$$
for $k \in [i+n,j-n]$.

Let $\eta = (\eta_k)_{k \in [i,j]}$ be a finite path in $G''$. There is then a collection $N(\eta)$ of finite paths $f = (f_k)_{k \in [i,j]}$ in $G$ such that $L_{G''}(\eta) = L_G(f)$ for all $f \in N(\eta)$ and
$$
N(\eta_k) = \left\{f_k : \ f \in N(\eta)\right\} 
$$ 
for all $k \in [i,j]$. Conversely, given a finite collection $N$ of paths $f = (f_k)_{k \in [i,j]}$ in $G$ such that $L_G(f) = L_G(f')$ for all $f,f' \in N$, there is a path $\eta = (\eta_k)_{k \in [i,j]}$ in $G''$ such that $N(\eta) = N$. With this identification in mind we make sense of the inclusion $\eta \subseteq \eta'$ when $\eta$ and $\eta'$ are paths in $G''$ over the same interval in $\mathbb Z$.

We fix now a natural number $\kappa \in \mathbb N$ which is a window size for $\phi,\phi^{-1},\psi$ and $\psi^{-1}$. When $j -i > 2\kappa$ the paths $\phi(f), f \in N(\eta),$ in $G$ are defined and we define the path $\phi''(\eta)$ in $H''$ such that 
$$
N(\phi''(\eta)) = \left\{\phi(f): \ f \in N(\eta) \right\}.
$$
Note that this construction is consistent with how $\phi''$ is defined on $X_{G''}$. That is, when $x \in X_{G''}$,
$$
\phi''(x)_{[i+\kappa,j-\kappa]} = \phi''(x_{[i,j]})
$$ 
when $i,j \in \mathbb Z, \ j-i > 2\kappa$.
Let $x \in X_{\overline{G}}$. An interval $I \subseteq \mathbb Z$ will be called \emph{a component interval in $x$} when there is a component $C$ in $\overline{G}$ such that $x_i \in E_C$ for all $i \in I$. Recall that such an interval is also homogeneous in $x$. Since a component in $\overline{G}$ is a component also in $G''$ it follows that $x_I$ is a path in $G''$ when $I$ is a component interval in $x$.

\begin{lemma}\label{15-02-26cx} Let $x \in X_{\overline{G}}$, and let $I:= [i,j]$ and $J:=[k,l]$ be component intervals in $x$. Assume that $j-i > 4\kappa$, $l-k > 2\kappa$ and $j < k$. There is a unique path $q\in X_{\overline{H}}[i+\kappa,l-\kappa]$ such that 
\begin{itemize}
\item[a)] $q_{[i+\kappa,j-\kappa]} = \phi''(x_{[i,j]})$, 
\item[b)] $q_{[k+\kappa,l-\kappa]} =\phi''(x_{[k,l]})$, and 
\item[c)] $L_{\overline{H}}(q) = \psi(L_{\overline{G}}(x))_{[i+\kappa,l-\kappa]}$.
\end{itemize}
\end{lemma}
\begin{proof} The uniqueness is obvious since $\overline{H}$ is right-resolving, and hence the existence is the only issue. Let $\hat{x}$ be the maximal $G''$-path dominated by $x_{[i,l]}$. Then $\phi''(\hat{x})_{[i+\kappa,j-\kappa]} \subseteq \phi''(x_{[i,j]})$ while 
\begin{equation}\label{17-02-26g}
\phi''(\hat{x})_{[k+\kappa,l-\kappa]} = \phi''(x_{[k,l]})
\end{equation} 
since $\hat{x}_{[k,l]} = x_{[k,l]}$ by (4) of Lemma \ref{25-03-26b}. It is tempting to deduce from Lemma \ref{12-02-26bx} that 
\begin{equation}\label{17-02-26dx}
L_{H''}(\phi''(\hat{x})) =\psi(L_{G''}(\hat{x}))_{[i+\kappa,l-\kappa]} = \psi(L_{\overline{G}}(x))_{[i+\kappa,l-\kappa]},
\end{equation}
and the conclusion is correct. But there is a slight problem here because $G''$ can contain sources as well as sinks and hence the path $\hat{x}$ may not sit in an element of $X_{G''}$, making Lemma \ref{12-02-26bx} useless. The way out goes as follows: Choose one of the paths $f' \in N(\hat{x})$ and note that since $G$ does not contain sources or sinks there is an element $f \in X_G$ such that $f' = f_{[i,l]}$. We can therefore use \eqref{09-10-25axy} in Lemma \ref{12-02-26bx} to get 
\begin{align*}
& L_{H''}(\phi''(\hat{x})) = L_H(\phi(f')) \\
&= L_H(\phi(f))_{[i+\kappa,l-\kappa]} = \psi(L_G(f))_{[i+\kappa,l-\kappa]} \\
& =\psi(L_G(f)_{[i,l]}) = \psi(L_G(f')_{[i,l]}) =  \psi(L_{G''}(\hat{x})_{[i,l]})
\\ 
&= \psi(L_{G''}(\hat{x}))_{[i+\kappa,l-\kappa]} =\psi(L_{\overline{G}}(x))_{[i+\kappa,l-\kappa]}.
\end{align*}
This establishes \eqref{17-02-26dx}.

Let $N$ be the collection of paths $f = (f_r)_{r \in [i+\kappa,k+\kappa-1]}$ in $H$ for which $s_{H}(f) \in s_{\overline{H}}(\phi''(x_{[i,j]})_{i + \kappa})$ and $L_H(f) = \psi(L_{\overline{G}}(x))_{[i+\kappa,k+\kappa-1]}$. Let
$\eta = (\eta_r)_{r\in [i+\kappa,k+\kappa-1]}$ be the corresponding path in $H''$; i.e. $N(\eta) =\left\{N(f) : \ f \in N\right\}$.
 Then
\begin{equation}\label{17-02-26e}
\phi''(\hat{x})_{[i+\kappa,j-\kappa]} \subseteq \eta_{[i+\kappa,j-\kappa]} \subseteq \phi''(x_{[i,j]}) .
\end{equation}
In particular, \eqref{17-02-26dx} and the first inclusion in \eqref{17-02-26e} imply that $N \neq \emptyset$. By definition of $\overline{H}$ this means that there is a path $q'$ in $X_{\overline{H}}[i+\kappa,k+\kappa -1]$ with $s_{\overline{H}}(q') = s_{\overline{H}}( \phi''(x_{[i,j]})_{i + \kappa})$ and label $L_{\overline{H}}(q')_{[i+\kappa,k+\kappa-1]} = \psi(L_{\overline{G}}(x))_{[i+\kappa,k+\kappa-1]}$. We note that
$$
q'_{[i+\kappa,j-\kappa]} = \phi''(x_{[i,j]})
$$
and
\begin{equation}\label{17-02-26f}
t_{\overline{H}}(q') = t_{H''}(\eta) .
\end{equation}
It follows from \eqref{17-02-26e} that 
\begin{equation}\label{17-02-26a}
\hat{x}_{[i+2\kappa,j-2\kappa]} \subseteq {\phi''}^{-1}(\eta)_{[i+2\kappa,j-2\kappa]}) \subseteq x_{[i+2\kappa,j-2\kappa]}. 
\end{equation}
Arguing as above, using again \eqref{09-10-25axy} instead of \eqref{17-11-25cx} in Lemma \ref{12-02-26bx}, we find that 
\begin{equation}\label{17-02-26b}
L_{G''}({\phi''}^{-1}(\eta))_{[i+2\kappa,k-1]} = \psi^{-1}(L_{H''}(\eta))_{[i+2\kappa,k-1]} = L_{\overline{G}}(x)_{[i+2\kappa,k-1]}.
\end{equation}
By Lemma \ref{13-02-26cx} $\hat{x}_{[i+2\kappa, k-1]}$ is the maximal $G''$-path dominated by $x_{[i+2\kappa,k-1]}$ and we conclude therefore from \eqref{17-02-26a} and \eqref{17-02-26b} that
$$
\hat{x}_{[i+2\kappa,k-1]} ={\phi''}^{-1}(\eta)_{[i+2\kappa,k-1]}.
$$
This implies that 
$$
\phi''(\hat{x})_{[i+3\kappa, k-\kappa-1]} = \eta_{[i+3\kappa,k-\kappa-1]},
$$
and since $\phi''(\hat{x})_{[i+3\kappa, k+\kappa-1]}$ and $\eta_{[i+3\kappa,k+\kappa-1]}$ have the same label, also that
$$
\phi''(\hat{x})_{[i+3\kappa, k+\kappa-1]} = \eta_{[i+3\kappa,k+\kappa-1]}.
$$
In particular, we can combine with \eqref{17-02-26f} and \eqref{17-02-26g} to conclude that 
$$
t_{\overline{H}}(q') = t_{H''}(\eta_{k+\kappa-1}) = t_{H''}(\phi''(\hat{x})_{k+\kappa-1}) = s_{\overline{H}}(\phi''(x_{[k,l]})_{k+\kappa}).
$$
The composition $q'\phi''(x_{[k,l]})$ is then a path in $\overline{H}$ with the stated properties.
\end{proof}

\emph{Proof of Theorem \ref{25-03-26}:} Let $x \in X_{\overline{G}}$. If $x \in \overline{\Per (X_{\overline{G}})}$, $x \in X_{G''}$ and we set $\overline{\phi}(x) := \phi''(x)$. Assume $x \notin \overline{\Per (X_{\overline{G}})}$. Let then $a_1 < a_2< \cdots < a_{2N}$ be the endpoints in $\mathbb Z$ of component intervals in $x$ of length exceeding $8\kappa$. Thus $x_{a_i}$ is in a component of $\overline{G}$ for all $i$ while $x_{a_i+1}$ is not when $i$ is odd and $x_{a_i-1}$ is not when $i$ is even. To define $\overline{\phi}(x)_i$ when 
$$
i \in [a_j-\kappa,a_{j+1}+\kappa], \ j = 1,3,5, \cdots ,2N-1 ,
$$
we note that $[a_j-7\kappa,a_j]$ and $[a_{j+1}, a_{j+1}+7\kappa]$ are component intervals in $x$ for which Lemma \ref{15-02-26cx} applies to give us a path $q \in X_{\overline{H}}[a_j-6\kappa,a_{j+1}+6\kappa]$ such that 
\begin{itemize}
\item[a)] $q_{[a_j-6\kappa,a_j-\kappa ]} = \phi''(x_{[a_j-7\kappa,a_j]})$, 
\item[b)] $q_{[a_{j+1}+\kappa,a_{j+1}+6\kappa]} = \phi''(x_{[a_{j+1}, a_{j+1}+7\kappa]})$, and \item[c)] $L_{\overline{H}}(q) = \psi(L_{\overline{G}}(x))_{[a_j-6\kappa,a_{j+1}+6\kappa]}$.
\end{itemize}
We set  
$$
\overline{\phi}(x)_i := q_i
$$ 
when $i \in \bigcup_{j=1}^N [a_{2j-1} -\kappa, a_{2j}+ \kappa]$. For each
$i \notin \bigcup_{j=1}^{N} [a_{2j-1}-\kappa,a_{2j}+\kappa]$,
the path $x_{[i-\kappa,i+\kappa]}$ is in a component of $\overline{G}$ and we set
$$
\overline{\phi}(x)_i := \phi''(x_{[i-\kappa,i+\kappa]}).
$$

Then $\overline{\phi}(x) \in X_{\overline{H}}$ and we have therefore defined a map $\overline{\phi} : X_{\overline{G}} \to  X_{\overline{H}}$. It is easy to see that $\overline{\phi}$ commutes with the shift, that $\overline{\phi}$ extends $\phi''|_{\overline{\Per (X_{\overline{G}}})}$ and that $L_{\overline{H}} \circ \overline{\phi} = \psi \circ L_{\overline{G}}$. To see that $\overline{\phi}$ is continuous and hence a sliding block code, note that since $\overline{G}$ is a finite directed graph there is an $D \in \mathbb N$ such that any path in $\overline{G}$ of length $D$ contains a path in a component of $\overline{G}$ of length exceeding $8\kappa$. It follows that $\overline{\phi}(x)_i$ is determined by $x_{[i-D,i+D]}$.

To prove that $\overline{\phi}$ extends $\phi''$, let $x \in X_{G''}$. There is an $N \in \mathbb Z$ and a component $C$ in $G''$ such that $x_i \in E_C$ for all $i \leq N$. Since $C$ is also a component in $\overline{G}$, it follows from the definition of $\overline{\phi}$ that $\overline{\phi}(x)_i = \phi''(x)_i$ for $i < N-\kappa$. Since 
 $$
  L_{\overline{H}}(\phi''(x)) =L_{H''}(\phi''(x)) =\psi( L_{G''}(x)) = \psi( L_{\overline{G}}(x)) = L_{\overline{H}}(\overline{\phi}(x))
 $$
 it follows that $\phi''(x) =\overline{\phi}(x)$ because $(\overline{H},L_{\overline{H}})$ is right-resolving.
 
To show that $\overline{\phi}$ is the unique sliding block code which extends $\phi''$ and makes the diagram in (2) of Theorem \ref{25-03-26} commute, assume that $\lambda :  X_{\overline{G}} \to X_{\overline{H}}$ is a sliding block code with these two properties. Let $x \in X_{\overline{G}}$. Since $x$ is backward asymptotic to an element of $\overline{\Per (X_{\overline{G}})} \subseteq X_{G''}$ it follows that there is a $K \in \mathbb Z$ such $\lambda(x)_j = \overline{\phi}(x)_j$ for $j \leq K$. Thanks to the commutativity of the diagram (2) and because $(\overline{H},L_{\overline{H}})$ is right-resolving, it follows that $\lambda(x) = \overline{\phi}(x)$. 
 
 It remains only to show that $\overline{\phi}$ is a conjugacy; not merely a sliding block code. For this note that since $\psi$ and $\phi$ are conjugacies there is also a commuting diagram
 \begin{equation*}
\xymatrix{
 X_H\ar[r]^-{\phi^{-1}} \ar[d]_{L_H}& X_{G} \ar[d]^-{L_G}\\
   Z \ar[r]^{\psi^{-1}} & Y}
\end{equation*}
 We get therefore also a sliding block code $\overline{\phi^{-1}} : X_{\overline{H}} \to X_{\overline{G}}$ with properties symmetric to those of $\overline{\phi}$. The composition $\overline{\phi^{-1}} \circ \overline{\phi} : X_{\overline{G}} \to X_{\overline{G}}$ must be the identity by the uniqueness properties we have established, applied with $\psi$ the identity map on $Y$ and $\phi$ the identity map on $X_G$. Thus $\overline{\phi^{-1}} \circ \overline{\phi} = \id_{X_{\overline{G}}}$ and by symmetry  $\overline{\phi} \circ \overline{\phi^{-1}} = \id_{X_{\overline{H}}}$.   
\qed

\section{Strongly canonical covers of the future cover}

\subsection{The main result}


\begin{thm}\label{18-01-26d} Let $(G,L_G)$ and $(H,L_H)$ be right-resolving presentations of the sofic subshifts $Y$ and $Z$, respectively. Assume that $\psi : Y \to Z$ and $\phi : X_G \to X_H$ are conjugacies such that
\begin{equation}\label{09-10-25axx}
\xymatrix{
 X_G\ar[r]^-{\phi} \ar[d]_{L_G}& X_{H} \ar[d]^-{L_H}\\
   Y \ar[r]^{\psi} & Z}
\end{equation}
commutes. There is a unique conjugacy $\underline{\phi} : X_{\underline{G}} \to X_{\underline{H}}$ such that
\begin{equation}\label{18-01-26e}
\xymatrix{
 X_{\underline{G}} \ar[r]^-{\underline{\phi}} \ar[d]_{L_{\underline{G}}}& X_{\underline{H}} \ar[d]^-{L_{\underline{H}}}\\
   Y \ar[r]_{\psi} & Z}
\end{equation}
commutes. This conjugacy has the properties that also
\begin{equation}\label{20-01-25d}
\xymatrix{\\
X_{\underline{G}} \ar[r]^-{\underline{\phi}} \ar[d]_-{f_{\underline{G}}} & X_{\underline{H}}  \ar[d]^-{f_{\underline{H}}}\\
 X_{\mathbb K(Y)} \ar[d]_-{L_{\mathbb K(Y)}} \ar[r]^-{\psi_{\mathbb K}} &  X_{\mathbb K(Z)} \ar[d]^-{L_{\mathbb K(Z)}} \\
Y \ar[r]_-{\psi}  & Z }
\end{equation}
and
\begin{equation}\label{09-02-26}
\xymatrix{
 X_{\underline{G}} \ar[r]^-{\underline{\phi}} & X_{\underline{H}} \\
   Y \ar[u]^{\alpha_G}\ar[r]_{\psi} & Z \ar[u]_{\alpha_H}}
\end{equation}
commute.
\end{thm}

The map $\psi_\mathbb K$ in the diagram \eqref{20-01-25d} is a conjugacy, and together with the lower commuting square 
in \eqref{20-01-25d} it comes from Krieger's theorem and requires only the existence of the conjugacy $\psi$. See Theorem 2.11 in \cite{Th1}. The novelty compared to Krieger's theorem is the conjugacy $\underline{\phi}$, and the rest of this section is dedicated to its construction and the proof of Theorem \ref{18-01-26d}.

It follows from Theorem \ref{18-01-26d} that when $(G,L_G)$ is a right-resolving and weakly canonical cover of $Y$ in the sense of \cite{Th1}, then $(\underline{G},L_{\underline{G}})$ is a strongly canonical cover of $Y$.



We assume now that we are in the setting of Theorem \ref{18-01-26d}. Note that since $(G,L_G)$ is right-resolving the set $L_G^{-1}(y)$ is finite for all $y\in Y$. By using the identification of elements in $X_{G''}$ with finite subsets of $X_G$ which was described following Lemma \ref{12-02-26bx} we can therefore define a map $\beta_G : Y \to X_{G''}$ such that
$$
N(\beta_G(y)) = L_G^{-1}(y) .
$$
Since \eqref{09-10-25axx} commutes it follows that
\begin{equation}\label{12-02-26d}
\phi'' \circ \beta_G = \beta_H \circ \psi .
\end{equation}
Together with the map $\alpha_G$ defined in Section \ref{consec} the map $\beta_G$ will be an important tool. 

\subsection{The non-wandering part}

\begin{lemma}\label{08-10-25bxz} For all $y \in Y$ the element $\beta_G(y)$ is forward asymptotic to $\alpha_G(y)$.
\end{lemma}
\begin{proof} Note that
$$
s_{\underline{G}}(\alpha_G(y)_n) = \left\{t_G(x): \ x \in X_G(-\infty,n-1], \ L_G(x) = y_{(-\infty,n-1]}\right\} \supseteq  s_{G''}(\beta_G(y)_n) 
$$
for all $n \in \mathbb Z$. A compactness argument shows that there is an $N \in \mathbb Z$ such that
$$
\left\{t_G(x): \ x \in X_G(-\infty,n-1], \ L_G(x) = y_{(-\infty,n-1]}\right\} = s_{G''}(\beta_G(y)_n) 
$$
for all $n \geq N$. It follows that $\alpha_G(y)_j = \beta_G(y)_j$ for all $j \geq N+1$.
\end{proof}

When $y$ is periodic, both $\alpha_G(y)$ and $\beta_G(y)$ are periodic and hence Lemma \ref{08-10-25bxz} has the following

\begin{corollary}\label{09-10-25b} Let $p \in \Per (Y)$. Then $\beta_G(p) = \alpha_G(p)$.
\end{corollary}

When $w \in A^m$ we denote by $w^\infty$ the $m$-periodic element $w^\infty$ in $\Per(A^\mathbb Z)$ with the property that $(w^\infty)_{[0,m-1]} = w$. The proof of the following lemma is similar to that of Lemma 2.16 in \cite{Th1}.

\begin{lemma}\label{10-02-25dxxz} Let $v\in V_{\underline{G}}$. There is a periodic point $p \in \Per (Y)$ and an element $y' \in X_{\underline{G}}$ backward asymptotic to $\alpha_G(p)$ such that $t_{\underline{G}}(y'_{(-\infty,-1]}) = v$.
\end{lemma}
\begin{proof} Since $v \in V_{\underline{G}}$ there is an element $y \in Y$ such that $v = s_{\underline{G}}(\alpha_G(y)_0)$. Set $D_j := s_{\underline{G}}(\alpha_G(y)_j), \ j \leq 0$. Since $(G,L_G)$ is right-resolving there is an $N <0$ such that 
\begin{equation}\label{17-04-25bxz}
\# D_j = \# D_N
\end{equation} 
for all $j \leq N$.
 By the pigeon hole principle there is a sequence $\cdots < n_3 < n_2 <n_1 < n_0 \leq N$ and a vertex $D \in V_{\underline{G}}$ such that $D =s_{\underline{G}}(\alpha_G(y)_{n_k})$ for all $k \in \mathbb N \cup \{0\}$. Consider the following loops in $\underline{G}$:
$$
\gamma_k := D \overset{y_{n_k}}{\to} D_{n_k+1}  \overset{y_{n_k+1}}{\to} \cdots \cdots \overset{y_{n_0-1}}{\to} D .
$$
If $\# L_G^{-1}(L_{\underline{G}}(\gamma_k)^\infty) > \# D$ for all $k$ it follows that there is an element $x' \in X_{\overline{G}}(-\infty,{n_0}-1]$ such that $D \subsetneq t_{\overline{G}}(x')$ and $L_{\overline{G}}(x') = y_{(-\infty,n_0-1]}$. By definition of $\alpha_G(y)$ this implies that $\# s_{\underline{G}}(\alpha_G(y)_j)  = \# D_j > \# D =\# D_N $ for all $j \leq n_0$, contradicting \eqref{17-04-25bxz}. This shows that there is a $k$ such that $q:= L_{\underline{G}}((\gamma_k)^\infty) \in Y$ is periodic and $\# L_G^{-1}(q) = \#D$. Thus $\alpha_G(q) = (\gamma_k)^\infty$ and 
$$
y' := \alpha_G(\sigma^{n_0}(q))_{(-\infty,n_0-1]}\alpha_G(y)_{[n_0,\infty)}
$$
has the stated properties.
\end{proof}

\begin{lemma}\label{11-11-24axz} Let $x,y \in X_{\underline{G}}$, and assume that $x = \alpha_G(L_{\underline{G}}(x))$. If there is an $N \in \mathbb Z$ such that $x_i = y_i$ for $i \leq N$, then $y = \alpha_G(L_{\underline{G}}(y))$.
\end{lemma}
\begin{proof} Since $L_{\underline{G}}(x)_i = L_{\underline{G}}(y)_i$ for $i \leq N$, it follows from the definition of 
$\alpha_G$ that 
$$
\alpha_G(L_{\underline{G}}(x))_j = \alpha_G(L_{\underline{G}}(y))_j
$$ 
for $j \leq N-1$. Hence 
$$
t_{\underline{G}}(y_{N-1}) =t_{\underline{G}}(x_{N-1}) = t_{\underline{G}}\left(\alpha_G(L_{\underline{G}}(x))_{N-1}\right) =  t_{\underline{G}}\left(\alpha_G(L_{\underline{G}}(y))_{N-1}\right).
$$
Since $L_{\underline{G}}(y_{[N,\infty)}) = L_{\underline{G}}\left(\alpha_G(L_{\underline{G}}(y))_{[N,\infty)}\right)$ it follows that $y_{[N,\infty)} = \alpha_G(L_{\underline{G}}(y))_{[N,\infty)}$ because $(\underline{G},L_{\underline{G}})$ is right-resolving.  We know already that $y_{(-\infty,N-1]} = x_{(-\infty,N-1]} = \alpha_G(L_{\underline{G}}(x))_{(-\infty,N-1]} =  \alpha_G(L_{\underline{G}}(y))_{(-\infty,N-1]}$, and conclude therefore that $y = \alpha_G(L_{\underline{G}}(y))$.
\end{proof}

\begin{lemma}\label{03-07-25axz} Let $x \in X_{\underline{G}}$. There is an element $y \in Y$ such that $\beta_G(y)$ is forward asymptotic to $x$.
\end{lemma}
\begin{proof} By Lemma \ref{10-02-25dxxz} there is an element $y' \in X_{\underline{G}}$ and a periodic point $p\in Y$ such that $y'$ is backward asymptotic to $\alpha_G(p)$ and forward asymptotic to $x$. Since $\alpha_G(p) =  \alpha_G(L_{\underline{G}}(\alpha_G(p)))$, it follows from Lemma \ref{11-11-24axz} that $y' = \alpha_G(L_{\underline{G}}(y'))$. Set $y := L_{\underline{G}}(y')$ and note that $\beta_G(y)$ is forward asymptotic to $\alpha_G(y) = y'$ by Lemma \ref{08-10-25bxz}. Hence $\beta_G(y)$ is also forward asymptotic to $x$.
\end{proof}

\begin{lemma}\label{08-10-25az}
$$
\overline{\Per (X_{\underline{G}}})
 = \bigcup_{y \in Y} \bigcap_{N \geq 1} \overline{\left\{\sigma^n(\beta_G(y)) : \ n \geq N \right\}} .
$$
\end{lemma}
\begin{proof} Let $y \in Y$. It follows from Lemma \ref{08-10-25bxz} that
$$
\bigcap_{N \geq 1} \overline{\left\{\sigma^n(\beta_G(y)) : \ n \geq N \right\}} = \bigcap_{N \geq 1} \overline{\left\{\sigma^n(\alpha_G(y)) : \ n \geq N \right\}}.
$$
Since there is a component of $X_{\underline{G}}$ which contains $\alpha_G(y)_n$ for all sufficiently large $n$, it follows that 
$$
 \bigcap_{N \geq 1} \overline{\left\{\sigma^n(\beta_G(y)) : \ n \geq N \right\}} \subseteq \overline{\Per (X_{\underline{G}})} .
 $$
This proves the inclusion $\supseteq$. The obtain the reversed inclusion, let $C$ be a component in $X_{\underline{G}}$. Choose an element $x \in X_C$ which is forward transitive in $X_C$; i.e. every word from $\mathbb W(X_C)$ occurs infinitely many times to the right in $x$. Let $d$ be a metric for the topology of $X_{\overline{G}}$. It follows from Lemma \ref{03-07-25axz} that there is an element $y \in Y$ such that 
\begin{equation}\label{08-10-25dz}
\lim_{n \to \infty} d\left(\sigma^n(\beta_G(y)),\sigma^n(x)\right) = 0.
\end{equation}
For each $z \in X_C$ there is a strictly increasing sequence $\{n_i\}$ in $\mathbb N$ such that
$\lim_{i \to \infty} d(\sigma^{n_i}(x),z) = 0$.
Thanks to \eqref{08-10-25dz}, by passing to a subsequence we can then arrange that
$$
\lim_{i \to \infty} d(\sigma^{n_i}(\beta_G(y)),z) = 0.
$$
It follows that $z \in  \bigcap_{N \geq 1} \overline{\left\{\sigma^n(\beta_G(y)) : \ n \geq N \right\}}$. Since $\overline{\Per(X_{\underline{G}})} = \bigcup_C X_C$, where we take the union over all components $C$ of $\underline{G}$, it follows that we have established the inclusion $\subseteq$.
\end{proof}



\begin{lemma}\label{09-10-25fz} Let $\phi'': X_{G''} \to X_{H''}$ be the conjugacy from Lemma \ref{12-02-26bx}. Then 
\begin{itemize}
\item[a)] $\phi''( \overline{\Per (X_{\underline{G}})}) = \overline{\Per (X_{\underline{H}})}$,
\item[b)]  $ f_{\underline{H}} \circ \phi''(x) = \psi_{\mathbb K} \circ f_{\underline{G}}(x)$ for all $x \in \overline{\Per (X_{\underline{G}})}$.

\end{itemize}
\end{lemma}

\begin{proof} It follows from \eqref{12-02-26d} and Lemma \ref{08-10-25az} that $\phi''( \overline{\Per (X_{\underline{G}})}) \subseteq \overline{\Per (X_{\underline{H}})}$. By symmetry we have equality, proving a).

 Let $z \in \overline{\Per (X_{\underline{G}})}$. By Lemma \ref{08-10-25az} there is an $y \in Y$ such that $z = \lim_{i \to \infty} \sigma^{n_i}(\beta_{G}(y)) $ for some increasing sequence $\{n_i\}$ in $\mathbb N$. Using again that $\phi'' \circ \beta_G = \beta_H \circ \psi$ we get
\begin{align*}
&\phi''( \lim_{i \to \infty} \sigma^{n_i}(\beta_{G}(y))) = \lim_{i \to \infty} \sigma^{n_i}(\beta_{H}(\psi(y))) =  \lim_{i \to \infty} \sigma^{n_i}(\alpha_{H}(\psi(y))),
\end{align*}
where the last identity follows from Lemma \ref{08-10-25bxz}. Note that $f_{\underline{H}} \circ \alpha_H  = \alpha_{Z}$ by Lemma \ref{27-01-26b}.  By Krieger's theorem, cf. (7) in Theorem 2.11 of \cite{Th1}, $\alpha_{Z}\circ\psi = \psi_{\mathbb K} \circ \alpha_Y$. We conclude therefore that
\begin{align*}
&f_{\underline{H}} \circ \phi''(z) = f_{\underline{H}} \circ \phi''( \lim_{i \to \infty} \sigma^{n_i}(\beta_{G}(y))) = \lim_{i \to \infty} f_{\underline{H}}(\sigma^{n_i}(\alpha_{H}(\psi(y)))) \\
&  = \lim_{i \to \infty} \sigma^{n_i}\left( f_{\underline{H}} \circ \alpha_H \circ \psi(y)\right)= \lim_{i \to \infty}  \sigma^{n_i}(\alpha_{Z}(\psi(y)))= \lim_{i \to \infty}  \sigma^{n_i}(\psi_{\mathbb K}\circ\alpha_Y(y)) \\
& =  \psi_{\mathbb K}\left(\lim_{i \to \infty} \sigma^{n_i}(\alpha_{Y}(y))\right) =  \psi_{\mathbb K}\left(\lim_{i \to \infty} \sigma^{n_i}(f_{\underline{G}} \circ \alpha_{G}(y))\right) = \psi_{\mathbb K}\circ f_{\underline{G}}(z),
\end{align*}
proving b).
\end{proof}

Before we continue the proof of Theorem \ref{18-01-26d} it may be informative to see in an example what the graphs $(\overline{G},L_{\overline{G}})$, $(\underline{G},L_{\underline{G}})$ and $(G'',L_{G''})$ may look like.

\begin{ex} \label{08-01-25a}

The following labeled graph $(G,L_G)$ is irreducible, right-resolving and follower-separated and hence the minimal right-resolving graph of the sofic subshift $Y$ it represents. Note that $Y$ is mixing.

\begin{equation}\label{17-07-26b}
 \begin{tikzpicture}[node distance={35mm}, thick, main/.style = {draw, circle}] 
\node[main] (1) {$a$}; 
\node[main] (2) [right of=1] {$b$}; 
\node[main] (3) [right of=2] {$c$};
\draw[->] (1) to [out=170,in=200,looseness=20] node[left ,pos=0.5] {$2$} (1); 
\draw[->] (1) to [out=80,in=100,looseness=2] node[above ,pos=0.5] {$1$} (2); 
\draw[->] (1) to [out=45,in=150,looseness=1.5] node[above ,pos=0.5] {$0$} (2); 
\draw[->] (2) to [out=260,in=280,looseness=2] node[above ,pos=0.5] {$1$} (1);
\draw[->] (2) to [out=230,in=300,looseness=1.5] node[above ,pos=0.5] {$0$} (1); 
\draw[->] (2) to [out=45,in=170,looseness=1] node[above ,pos=0.5] {$2$} (3); 
\draw[->] (3) to [out=200,in=350,looseness=1] node[below ,pos=0.5] {$3$} (2);
\draw[->] (3) to [out=45,in=350,looseness=10] node[above,pos=0.5] {$0$} (3);  
\draw[->] (3) to [out=70,in=290,looseness=30] node[right,pos=0.5] {$1$} (3); 
\end{tikzpicture} 
\end{equation}

In this example the graph $(\overline{G},L_{\overline{G}})$ becomes the following.

\begin{equation*}
 \begin{tikzpicture}[node distance={35mm}, thick, main/.style = {draw, circle}] 
\node[main] (1) {$\{a\}$}; 
\node[main] (2) [right of=1] {$\{b\}$}; 
\node[main] (3) [right of=2] {$\{c\}$};
\node[main] (4) [below of=2] {$\{a,c\}$};
\node[main] (5) [right of=4] {$\{b,c\}$};
\node[main] (6) [below of=5] {$\{a,b,c\}$};
\node[main] (7) [below of=4] {$\{a,b\}$};
\draw[->] (1) to [out=170,in=200,looseness=20] node[left ,pos=0.5] {$2$} (1); 
\draw[->] (1) to [out=80,in=100,looseness=2] node[above ,pos=0.5] {$1$} (2); 
\draw[->] (1) to [out=45,in=150,looseness=1.5] node[above ,pos=0.5] {$0$} (2); 
\draw[->] (2) to [out=260,in=280,looseness=2] node[above ,pos=0.5] {$1$} (1);
\draw[->] (2) to [out=230,in=300,looseness=1.5] node[above ,pos=0.5] {$0$} (1); 
\draw[->] (2) to [out=45,in=170,looseness=1] node[above ,pos=0.5] {$2$} (3); 
\draw[->] (3) to [out=200,in=350,looseness=1] node[above ,pos=0.5] {$3$} (2);
\draw[->] (3) to [out=45,in=350,looseness=10] node[above,pos=0.5] {$0$} (3);  
\draw[->] (3) to [out=60,in=290,looseness=15] node[right,pos=0.5] {$1$} (3); 
\draw[->] (6) to [out=260,in=290,looseness=10] node[below,pos=0.5] {$0$} (6); 
\draw[->] (6) to [out=10,in=350,looseness=10] node[right,pos=0.5] {$1$} (6); 
\draw[->] (6) to [out=100,in=300,looseness=0.5] node[left,pos=0.5] {$2$} (4); 
\draw[->] (6) to [out=50,in=60,looseness=2.5] node[right,pos=0.5] {$3$} (2); 
\draw[->] (4) to [out=40,in=150,looseness=1.5] node[above,pos=0.5] {$0$} (5); 
\draw[->] (4) to [out=20,in=170,looseness=1] node[above,pos=0.5] {$1$} (5); 
\draw[->] (5) to [out=190,in=350,looseness=0.5] node[below,pos=0.5] {$0$} (4); 
\draw[->] (5) to [out=220,in=320,looseness=1] node[below,pos=0.3] {$1$} (4); 
\draw[->] (4) to [out=170,in=210,looseness=1] node[below,pos=0.3] {$2$} (1); 
\draw[->] (7) to [out=90,in=270,looseness=1] node[left,pos=0.5] {$2$} (4); 
\draw[->] (7) to [out=260,in=290,looseness=10] node[below,pos=0.5] {$0$} (7);
\draw[->] (7) to [out=170,in=200,looseness=10] node[left,pos=0.5] {$1$} (7);
\draw[->] (4) to [out=90,in=270,looseness=1] node[left,pos=0.4] {$3$} (2);
\draw[->] (5) to [out=90,in=270,looseness=1] node[left,pos=0.5] {$2$} (3);
\draw[->] (5) to [out=100,in=300,looseness=1] node[left,pos=0.5] {$3$} (2);
\end{tikzpicture} 
\end{equation*}
The graph $(\underline{G},L_{\underline{G}})$ is the following labeled subgraph of  $(\overline{G},L_{\overline{G}})$. We note that it is follower-separated and hence, by Corollary \ref{04-09-25d}, also a copy of the future cover of $Y$.

\begin{equation}\label{17-07-26d}
 \begin{tikzpicture}[node distance={35mm}, thick, main/.style = {draw, circle}] 
\node[main] (1) {$\{a\}$}; 
\node[main] (2) [right of=1] {$\{b\}$}; 
\node[main] (3) [right of=2] {$\{c\}$};
\node[main] (4) [below of=2] {$\{a,c\}$};
\node[main] (5) [right of=4] {$\{b,c\}$};
\node[main] (6) [below of=5] {$\{a,b,c\}$};
\draw[->] (1) to [out=170,in=200,looseness=20] node[left ,pos=0.5] {$2$} (1); 
\draw[->] (1) to [out=80,in=100,looseness=2] node[above ,pos=0.5] {$1$} (2); 
\draw[->] (1) to [out=45,in=150,looseness=1.5] node[above ,pos=0.5] {$0$} (2); 
\draw[->] (2) to [out=260,in=280,looseness=2] node[above ,pos=0.5] {$1$} (1);
\draw[->] (2) to [out=230,in=300,looseness=1.5] node[above ,pos=0.5] {$0$} (1); 
\draw[->] (2) to [out=45,in=170,looseness=1] node[above ,pos=0.5] {$2$} (3); 
\draw[->] (3) to [out=200,in=350,looseness=1] node[above ,pos=0.5] {$3$} (2);
\draw[->] (3) to [out=45,in=350,looseness=10] node[above,pos=0.5] {$0$} (3);  
\draw[->] (3) to [out=60,in=290,looseness=18] node[right,pos=0.5] {$1$} (3); 
\draw[->] (6) to [out=260,in=290,looseness=10] node[below,pos=0.5] {$0$} (6); 
\draw[->] (6) to [out=10,in=350,looseness=10] node[right,pos=0.5] {$1$} (6); 
\draw[->] (6) to [out=100,in=300,looseness=0.5] node[left,pos=0.5] {$2$} (4); 
\draw[->] (6) to [out=50,in=60,looseness=3] node[right,pos=0.5] {$3$} (2); 
\draw[->] (4) to [out=40,in=150,looseness=1.5] node[above,pos=0.5] {$0$} (5); 
\draw[->] (4) to [out=20,in=170,looseness=1] node[above,pos=0.5] {$1$} (5); 
\draw[->] (5) to [out=190,in=350,looseness=0.5] node[below,pos=0.5] {$0$} (4); 
\draw[->] (5) to [out=220,in=320,looseness=1] node[below,pos=0.3] {$1$} (4); 
\draw[->] (4) to [out=170,in=210,looseness=1] node[below,pos=0.3] {$2$} (1); 
\draw[->] (4) to [out=90,in=270,looseness=1] node[left,pos=0.4] {$3$} (2);
\draw[->] (5) to [out=90,in=270,looseness=1] node[left,pos=0.5] {$2$} (3);
\draw[->] (5) to [out=100,in=300,looseness=1] node[left,pos=0.5] {$3$} (2);
\end{tikzpicture} 
\end{equation}


In this example $(G'',L_{G''})$ is the following labeled graph.

\begin{equation*}
 \begin{tikzpicture}[node distance={35mm}, thick, main/.style = {draw, circle}] 
\node[main] (1) {$\{a\}$}; 
\node[main] (2) [right of=1] {$\{b\}$}; 
\node[main] (3) [right of=2] {$\{c\}$};
\node[main] (4) [below of=2] {$\{a,c\}$};
\node[main] (5) [right of=4] {$\{b,c\}$};
\node[main] (6) [below of=5] {$\{a,b,c\}$};
\node[main] (7) [below of=4] {$\{a,b\}$};
\draw[->] (1) to [out=170,in=200,looseness=20] node[left ,pos=0.5] {$2$} (1); 
\draw[->] (1) to [out=80,in=100,looseness=2] node[above ,pos=0.5] {$1$} (2); 
\draw[->] (1) to [out=45,in=150,looseness=1.5] node[above ,pos=0.5] {$0$} (2); 
\draw[->] (2) to [out=260,in=280,looseness=2] node[above ,pos=0.5] {$1$} (1);
\draw[->] (2) to [out=230,in=300,looseness=1.5] node[above ,pos=0.5] {$0$} (1); 
\draw[->] (2) to [out=45,in=170,looseness=1] node[above ,pos=0.5] {$2$} (3); 
\draw[->] (3) to [out=200,in=350,looseness=1] node[above ,pos=0.5] {$3$} (2);
\draw[->] (3) to [out=45,in=350,looseness=10] node[above,pos=0.5] {$0$} (3);  
\draw[->] (3) to [out=70,in=290,looseness=30] node[right,pos=0.5] {$1$} (3); 
\draw[->] (6) to [out=260,in=290,looseness=10] node[below,pos=0.5] {$0$} (6); 
\draw[->] (6) to [out=10,in=350,looseness=10] node[right,pos=0.5] {$1$} (6); 
\draw[->] (4) to [out=40,in=150,looseness=1.5] node[above,pos=0.5] {$0$} (5); 
\draw[->] (4) to [out=20,in=170,looseness=1] node[above,pos=0.5] {$1$} (5); 
\draw[->] (5) to [out=190,in=350,looseness=0.5] node[below,pos=0.5] {$0$} (4); 
\draw[->] (5) to [out=220,in=320,looseness=1] node[below,pos=0.3] {$1$} (4); 
\draw[->] (7) to [out=90,in=270,looseness=1] node[left,pos=0.5] {$2$} (4); 
\draw[->] (7) to [out=260,in=290,looseness=10] node[below,pos=0.5] {$0$} (7);
\draw[->] (7) to [out=170,in=200,looseness=10] node[left,pos=0.5] {$1$} (7);
\end{tikzpicture} 
\end{equation*}

We note that none of the two graphs, $\underline{G}$ or $G''$, is contained in the other, and that $X_{G''}$ does not factor onto the future cover $X_{\mathbb K(Y)}$ of $Y$. 
\end{ex}

In \cite{Th1} two strongly canonical covers $\mathbb K'(Y)$ and $\mathbb F'(Y)$ of a sofic shift $Y$ were introduced; albeit the second only for irreducible sofic subshifts. It can be checked that for the mixing sofic subshift $Y$ of Example \ref{08-01-25a} neither $X_{\mathbb F'(Y)}$ nor $X_{\mathbb K'(Y)}$ factor onto the future cover $X_{\mathbb K(Y)}$, justifying a claim made in the introduction.

\subsection{Completing the proof of the main result, Theorem \ref{18-01-26d}}

 A component $C$ in $\underline{G}$ is a \emph{source component} when it has the property that any path in $\underline{G}$ which terminates in $C$ must start in $C$. 

\begin{lemma}\label{20-01-26dx} Let $C$ be a source component in $\underline{G}$. Then
$$
\# L_G^{-1}\left(L_{\underline{G}}(x)\right) = M(C)
$$
for all $x \in X_C$.
\end{lemma}
\begin{proof} Since $C$ is a source component it follows from Lemma \ref{10-02-25dxxz} that there is a periodic point $p \in \Per (Y)$ such that $\alpha_G(p) \in X_C$. Since $\alpha_G(p) = \beta_G(p)$ by Corollary \ref{09-10-25b} it follows that $\# L_G^{-1}(p) = \# s_{\underline{G}}(\alpha_G(p)_j)$ for all $j \in \mathbb Z$ and hence that
\begin{equation}\label{20-01-26ex}
 \# L_G^{-1}(p) = M(C).
\end{equation} 
Assume for a contradiction that there is an $x \in X_C$ such that 
\begin{equation}\label{20-01-26fx}
\# L_G^{-1}(L_{\underline{G}}(x)) > M(C). 
\end{equation}
Since $C$ is strongly connected there is an element $z \in X_C$ and an $m \in \mathbb N$ such that $z_i = x_i$ for $i \leq 0$ while $z_i = \alpha_G(p)_i$ for $i \geq m$. It follows from \eqref{20-01-26fx} and the definition of $\alpha_G$ that for some $N \leq 0$
\begin{equation}\label{20-01-26gx}
\# s_{\underline{G}}(\alpha_G(L_{\underline{G}}(z))_i) \geq M(C)+1
\end{equation} 
for all $i \leq N$. Since $L_{\underline{G}}(z)_i = p_i$ for $i \geq m$, it follows from \eqref{20-01-26ex} and the periodicity of $p$ that $\#s_{\underline{G}}(\alpha_G(L_{\underline{G}}(z))_i) \leq M(C)$ for all sufficiently large $i$. However, since $s_{\underline{G}}(\alpha_G(L_{\underline{G}}(z))_i)  \supseteq s_{\underline{G}}(z_i)$ and $\# s_{\underline{G}}(z_i) = M(C)$ for all $i$, we find that $s_{\underline{G}}(\alpha_G(L_{\underline{G}}(z))_i)  =s_{\underline{G}}(z_i)$ and hence also that $\alpha_G(L_{\underline{G}}(z))_i = z_i$ for all sufficiently large $i$. Since $C$ is a source component it follows that $s_{\underline{G}}(\alpha_G(L_{\underline{G}}(z))_i) \in V_C$ for all $i$; a contradiction to \eqref{20-01-26gx}. Hence $\# L_G^{-1}(L_{\underline{G}}(x)) = M(C)$ for all $x \in X_C$. 
\end{proof}

\begin{cor}\label{20-01-26h} Let $C,C'$ be source components in $\underline{G}$. Then
\begin{itemize}
\item[1)] $z = \beta_G(L_{\underline{G}}(z))$ for all $z \in X_C$,
\item[2)] $L_{\underline{G}} : X_C \to Y$ is injective and hence a conjugacy of $X_C$ onto $L_{\underline{G}}(X_C)$, and
\item[3)] $L_{\underline{G}}(X_C) \cap L_{\underline{G}}(X_{C'}) = \emptyset$ if $C \neq C'$.
\end{itemize}
\end{cor}
\begin{proof} This follows readily from Lemma \ref{20-01-26dx}.
\end{proof}

\begin{lemma}\label{16-02-26d} For each $y \in Y$ there is a $z\in Y$ backward asymptotic to $y$ such that $\alpha_G(y)$ is backward asymptotic to $\beta_G(z)$.
\end{lemma}
\begin{proof} There is a component $C$ in $\underline{G}$ such that $\alpha_G(y)_j \in E_C$ for all $j \leq N$. Choose an element $w\in X_C$ such that $w_N = \alpha_G(y)_N$ and set
$$
z'_j = \begin{cases} \alpha_G(y)_j, \ j \leq N, \\ w_j, \ j > N. \end{cases}
$$
Set $z := L_{\underline{G}}(z')$. Then $z_j = y_j$ for $j \leq N$; in particular, $z$ is backward asymptotic to $y$. Furthermore, $\alpha_G(y)_j = z'_j \subseteq \beta_G(L_{\underline{G}}(z'))_j \subseteq \alpha_G(L_{\underline{G}}(z'))_j = \alpha_G(y)_j$ for $j \leq N$, proving that $\beta_G(z)_j = \alpha_G(y)_j$ for $j \leq N$ and that $\alpha_G(y)$ is backward asymptotic to $\beta_G(z)$.

\end{proof}


We are now ready for \emph{the proof of Theorem \ref{18-01-26d}:} It follows from Theorem \ref{25-03-26} that there is a conjugacy $\overline{\phi} : X_{\overline{G}} \to X_{\overline{H}}$ with the properties specified in (1) and (2) of that theorem. To show that $\overline{\phi}$ takes $X_{\underline{G}}$ onto $X_{\underline{H}}$ the first step is to show that 
\begin{equation}\label{26-04-26a}
\overline{\phi} \circ \alpha_G = \alpha_H \circ \psi.
\end{equation} 
Let $y \in Y$. By Lemma \ref{16-02-26d} there is a $z\in Y$ backward asymptotic to $y$ such that $\alpha_G(y)$ is backward asymptotic to $\beta_G(z)$. Hence $\overline{\phi}(\alpha_G(y))$ is backward asymptotic to $\overline{\phi}(\beta_G(z))$. Since $\beta_G(z)$ is backward asymptotic to an element of $\overline{\Per (X_{\underline{G}})}$ where $\overline{\phi}$ agrees with $\phi''$, we see that $\overline{\phi}(\alpha_G(y))$ is backward asymptotic to $\phi''(\beta_G(z)) = \beta_H(\psi(z))$, where the last equality comes from \eqref{12-02-26d}. Since $\beta_H(\psi(z))_j \subseteq \alpha_H(\psi(z))_j$ for all $j \in \mathbb Z$, it follows that $\overline{\phi}(\alpha_G(y))_j \subseteq \alpha_H(\psi(z))_j = \alpha_H(\psi(y))_j$ for all $j$ 'close to $-\infty$'. By symmetry $(\overline{\phi})^{-1}(\alpha_H(\psi(y)))_j \subseteq \alpha_G(y)_j$ for all $j$ 'close to $-\infty$' and hence
$$
\alpha_H(\psi(y))_j = \overline{\phi}((\overline{\phi})^{-1}(\alpha_H(\psi(y))))_j \subseteq \overline{\phi}(\alpha_G(y))_j
$$
for $j$ 'close to $-\infty$'. We conclude that
$$
\alpha_H(\psi(y))_j  = \overline{\phi}(\alpha_G(y))_j
$$
for $j$ 'close to $-\infty$'. Since $L_{\overline{H}}(\overline{\phi}(\alpha_G(y))) = \psi(y) =L_{\overline{H}}(\alpha_H(\psi(y)))$ we conclude that
$\overline{\phi}(\alpha_G(y)) =\alpha_H(\psi(y))$, proving \eqref{26-04-26a}.

 It follows from Lemma \ref{10-02-25dxxz} that elements backward asymptotic to an element of $\alpha_G(\Per(Y))$ are dense in $X_{\underline{G}}$. It follows from \eqref{26-04-26a} that $\overline{\phi}$ maps such elements to elements of $X_{\overline{H}}$ that are backward asymptotic to elements from $\alpha_H(\Per(Z)) \subseteq X_{\underline{H}}$. Hence $\overline{\phi}(X_{\underline{G}}) \subseteq X_{\underline{H}}$. Using this conclusion on $\overline{\phi}^{-1}$ we find that $\overline{\phi}(X_{\underline{G}}) = X_{\underline{H}}$. Set $\underline{\phi}:= \overline{\phi}|_{X_{\underline{G}}}$. Then the commutativity of \eqref{18-01-26e} follows from Theorem \ref{25-03-26}.

To prove uniqueness of $\underline{\phi}$ assume that a conjugacy $\lambda : X_{\underline{G}} \to X_{\underline{H}}$ can substitute for $\overline{\phi}$ in the diagram \eqref{18-01-26e}. Let $x \in X_{\underline{G}}$. If $x \in X_C$ for some source component $C$ in $\underline{G}$ we have that $\lambda(x) \in X_{C_1}$ and $\underline{\phi}(x) \in X_{C_2}$ for some source components $C_1, C_2$ in $\underline{H}$ and $L_{\underline{H}}(\lambda(x)) = \psi(L_{\underline{G}}(x)) = L_{\underline{H}}(\underline{\phi}(x))$. By Corollary \ref{20-01-26h} this implies that $\lambda(x) = \underline{\phi}(x)$. Since $(\underline{H},L_{\underline{H}})$ is right-resolving it follows readily that the same holds when $x$ is backward asymptotic to a source component in $\underline{G}$. Since every element of $X_{\underline{G}}$ can be approximated by elements that are backward asymptotic to a source component in $\underline{G}$ it follows by continuity that $\lambda = \underline{\phi}$.

To complete the proof it remains only to establish the commutativity of \eqref{20-01-25d}. The lower commuting square in \eqref{20-01-25d} comes from Krieger's theorem, Theorem 2.11 in \cite{Th1}, so it remains only to show that $\psi_{\mathbb K} \circ f_{\underline{G}} = f_{\underline{H}} \circ \underline{\phi}$. For this note that it follows from b) in Lemma \ref{09-10-25fz} that $\psi_{\mathbb K} \circ f_{\underline{G}}(x) = f_{\underline{H}} \circ \underline{\phi}(x)$ when $x \in \overline{\Per (X_{\underline{G}})}$. By Krieger's theorem $\psi \circ L_{\mathbb K(Y)} = L_{\mathbb K(Z)} \circ \psi_{\mathbb K}$, and hence
$$
L_{\mathbb K(Z)} \circ \psi_{\mathbb K} \circ f_{\underline{G}} = \psi \circ L_{\mathbb K(Y)}  \circ f_{\underline{G}} = \psi \circ L_{\underline{G}} = L_{\underline{H}} \circ \underline{\phi} = L_{\mathbb K(Z)} \circ f_{\underline{H}} \circ \underline{\phi}. 
$$
As $(\mathbb K(Z),L_{\mathbb K(Z)})$ is right-resolving and every element of $X_{\underline{G}}$ is backward asymptotic to an element of $ \overline{\Per ( X_{\underline{G}})}$, it follows from this that $\psi_{\mathbb K} \circ f_{\underline{G}} = f_{\underline{H}} \circ \underline{\phi}$. 
\qed

\begin{rem}\label{27-04-26} We note, for the record, that by construction the restriction of $\underline{\phi}$ to $X_{G''} \cap X_{\underline{G}}$ agrees with $\phi''$, and
$\underline{\phi}(X_{G''}  \cap X_{\underline{G}}) =  X_{H''}  \cap X_{\underline{H}}$. 
In particular, $\underline{\phi} = \phi'' = \overline{\phi}$ on $\overline{\Per(X_{\underline{G}})}$.
\end{rem}

\section{The extended future cover}

For a sofic subshift $Y$, set
$$
(\underline{\mathbb K}(Y),L_{\underline{\mathbb K}(Y)})  := (\underline{\mathbb K(Y)}, L_{\underline{\mathbb K(Y)}}).
$$
This is what we will call the extended future cover of $Y$.

By Krieger's theorem, as formulated in Theorem 2.11 of \cite{Th1}, and by Theorem \ref{18-01-26d} above, the extended future cover defines a canonical cover for sofic subshifts; strongly canonical in the sense of \cite{Th1}:

\begin{thm}\label{21-01-26} Let $\psi : Y \to Z$ be a conjugacy of sofic subshifts.  There is a unique conjugacy $\psi_{\underline{\mathbb K}} : X_{\underline{\mathbb K}(Y)} \to X_{\underline{\mathbb K}(Z)}$ such that
\begin{equation*}
\xymatrix{
 X_{\underline{\mathbb K}(Y)} \ar[r]^-{\psi_{\underline{\mathbb K}}} \ar[d]_{L_{\underline{\mathbb K}(Y)}}& X_{\underline{\mathbb K}(Z)} \ar[d]^-{L_{\underline{\mathbb K}(Z)}}\\
   Y \ar[r]^{\psi} & Z}
\end{equation*}
commutes. This conjugacy has the property that also
\begin{equation}\label{20-01-25dx}
\xymatrix{
 X_{\underline{\mathbb K}(Y)} \ar[r]^-{\psi_{\underline{\mathbb K}}} \ar[d]_{f_{\K(Y)}}& X_{\underline{\mathbb K}(Z)} \ar[d]^-{f_{\K(Z)}}\\
 X_{\mathbb K(Y)} \ar[d]_-{L_{\mathbb K(Y)}} \ar[r]^-{\psi_{\mathbb K}} &  X_{\mathbb K(Z)} \ar[d]^-{L_{\mathbb K(Z)}} \\
Y \ar[r]_-{\psi}  & Z }
\end{equation}
and
\begin{equation}\label{04-08-26}
\xymatrix{
 X_{\underline{\mathbb K}(Y)} \ar[r]^-{\psi_{\underline{\mathbb K}}} & X_{\underline{\mathbb K}(Z)}\\
   Y \ar[u]^-{\alpha_{\mathbb K(Y)}} \ar[r]^{\psi} & Z\ar[u]_-{\alpha_{\mathbb K(Z)}}}
\end{equation}
commute.
\end{thm}

We note that by Lemma \ref{27-01-26b} the right-inverse $\alpha_{\mathbb K(Y)} : Y \to X_{\underline{\mathbb K}(Y)}$ to the labeling $L_{\underline{\mathbb K}(Y)} :  X_{\underline{\mathbb K}(Y)} \to Y$ is a lift of the right-inverse $\alpha_Y : Y \to X_{\mathbb K(Y)}$ to the labeling  $L_{{\mathbb K}(Y)} :  X_{{\mathbb K}(Y)} \to Y$ in the sense that
$$
f_{\K(Y)} \circ \alpha_{\mathbb K(Y)} = \alpha_Y.
$$

\begin{ex}\label{21-02-26} In many cases the map $ f_{\K(Y)}$ in \eqref{20-01-25dx} is a conjugacy because $(\underline{\mathbb K}(Y),L_{\underline{\mathbb K}(Y)})$ turns out to be follower-separated, and then $(\underline{\mathbb K}(Y),L_{\underline{\mathbb K}(Y)}) = ({\mathbb K}(Y),L_{{\mathbb K}(Y)})$.
This is for example the case when $Y$ is the even shift which was considered in Example 4.9 of \cite{Th1} as well as in Example \ref{08-01-25a} above, so it seems appropriate to exhibit an example where this is not the case.
The following labeled graph is irreducible, right-resolving and follower-separated and hence the minimal right-resolving graph of the sofic subshift $Y$ it presents. Note that $Y$ is mixing.

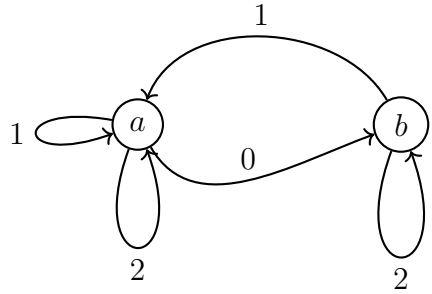
\begin{equation}\label{02-11-25d}
 \begin{tikzpicture}[node distance={35mm}, thick, main/.style = {draw, circle}] 
\node[main] (1) {$a$}; 
\node[main] (2) [right of=1] {$b$}; 
\draw[->] (1) to [out=170,in=200,looseness=20] node[left ,pos=0.5] {$1$} (1);
\draw[->] (1) to [out=250,in=290,looseness=20] node[below ,pos=0.5] {$2$} (1);
\draw[->] (1) to [out=300,in=200,looseness=1] node[above ,pos=0.5] {$0$} (2);
\draw[->] (2) to [out=120,in=70,looseness=1] node[above ,pos=0.5] {$1$} (1);
\draw[->] (2) to [out=250,in=290,looseness=20] node[below ,pos=0.5] {$2$} (2);
 \end{tikzpicture} 
\end{equation}
The graph $(\underline{\mathbb K}(Y),L_{\underline{\mathbb K}(Y)})$ is the  following.

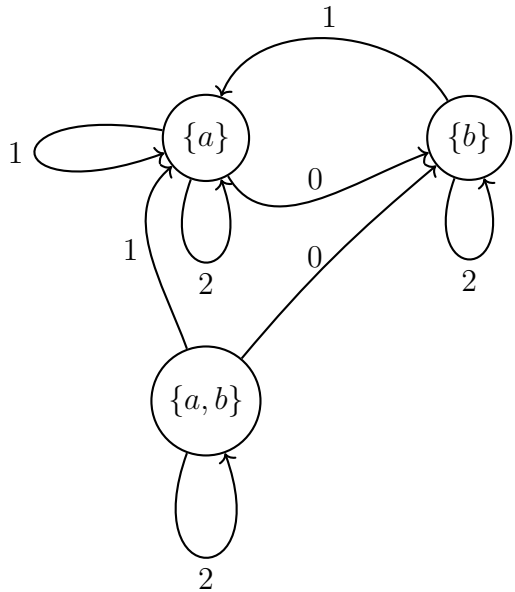
\begin{equation*}
 \begin{tikzpicture}[node distance={35mm}, thick, main/.style = {draw, circle}] 
\node[main] (1) {$\{a\}$}; 
\node[main] (2) [right of=1] {$\{b\}$};
\node[main] (3) [below of=1] {$\{a,b\}$}; 
\draw[->] (1) to [out=170,in=200,looseness=20] node[left ,pos=0.5] {$1$} (1);
\draw[->] (1) to [out=250,in=290,looseness=10] node[below ,pos=0.5] {$2$} (1);
\draw[->] (1) to [out=300,in=200,looseness=1] node[above ,pos=0.5] {$0$} (2);
\draw[->] (2) to [out=120,in=70,looseness=1] node[above ,pos=0.5] {$1$} (1);
\draw[->] (2) to [out=250,in=290,looseness=10] node[below ,pos=0.5] {$2$} (2);
\draw[->] (3) to [out=250,in=290,looseness=10] node[below ,pos=0.5] {$2$} (3);
\draw[->] (3) to [out=110,in=220,looseness=1] node[left ,pos=0.5] {$1$} (1);
\draw[->] (3) to [out=50,in=220,looseness=1] node[left ,pos=0.5] {$0$} (2);
 \end{tikzpicture} 
\end{equation*}
Here the vertex $\{a,b\}$ has the same follower set as $\{a\}$ and the merged graph of $(\underline{\mathbb K}(Y),L_{\underline{\mathbb K}(Y)})$, which is the future cover $(\mathbb K(Y),L_{\mathbb K(Y)})$ by Corollary \ref{04-09-25d}, is the graph \eqref{02-11-25d} we started from.

This example of a mixing sofic subshift whose extended future cover is larger than the future cover itself seems to be among the 'smallest' such examples. With larger graphs and more labels the difference can be much bigger. To illustrate this, and at the same time give an example of a mixing sofic subshift for which the minimal right-resolving presentation, the future cover and the extended future cover are all different, we invite the reader to check that this is the case when $Y$ is the sofic subshift whose minimal right-resolving presentation is given by the following labeled graph.

\begin{equation*}\label{02-11-24e}
 \begin{tikzpicture}[node distance={35mm}, thick, main/.style = {draw, circle}] 
\node[main] (1) {$a$}; 
\node[main] (2) [right of=1] {$b$};
\node[main] (3) [below of=1] {$c$}; 
\draw[->] (1) to [out=90,in=45,looseness=10] node[above ,pos=0.5] {$1$} (1);
\draw[->] (1) to [out=100,in=130,looseness=13] node[above ,pos=0.5] {$2$} (1);
\draw[->] (1) to [out=150,in=190,looseness=13] node[left ,pos=0.5] {$3$} (1);
\draw[->] (1) to [out=300,in=200,looseness=0.5] node[below ,pos=0.5] {$0$} (2);
\draw[->] (2) to [out=170,in=5,looseness=0.5] node[above ,pos=0.5] {$1$} (1);
\draw[->] (3) to [out=110,in=230,looseness=0.5] node[left ,pos=0.5] {$1$} (1);
\draw[->] (1) to [out=280,in=70,looseness=0.5] node[right ,pos=0.5] {$4$} (3);
\draw[->] (3) to [out=130,in=180,looseness=10] node[left ,pos=0.5] {$3$} (3);
\draw[->] (3) to [out=330,in=280,looseness=10] node[right ,pos=0.5] {$5$} (3);
\draw[->] (2) to [out=0,in=60,looseness=10] node[above,pos=0.5] {$2$} (2);
\draw[->] (2) to [out=260,in=330,looseness=10] node[right,pos=0.5] {$3$} (2);
\draw[->] (2) to [out=220,in=10,looseness=0.7] node[below,pos=0.5] {$4$} (3);
\end{tikzpicture} 
\end{equation*}

\end{ex}

\section{Desynchronization of sofic subshifts}
Let $Y$ be a sofic subshift. Recall that a word $w \in \mathbb W(Y)$ is \emph{synchronizing} when 
$u,v \in \mathbb W(Y), \ uw,wv \in \mathbb W(Y) \ \Rightarrow \ uwv \in \mathbb W(Y)$. We let $\sync(Y)$ denote the set of elements of $Y$ that contain a synchronizing word for $Y$, and set
$$
\eth Y := Y \backslash \sync(Y). 
$$
These sets, $\sync(Y)$ and $\eth Y$, are shift-invariant with $\sync(Y)$ open and $\eth Y$ closed. We call $\eth Y$ the \emph{desynchronization} of $Y$. It is well-known that $\sync(Y) \neq \emptyset$ when $Y$ is irreducible and hence that $\eth Y$ is a proper subshift of $Y$ in this case. It is also known from Theorem 6.6 of \cite{Th0} that $\eth Y$ is sofic when $Y$ is irreducible. The first goal here is to extend both conclusions to a general sofic subshift.

Let $(H,L_H)$ be a right-resolving labeled graph presenting the sofic subshift $Y$. Inspired by Definition 9.1.4 in \cite{LM} we say that a word $w \in \mathbb W(Y)$ is \emph{magic} for $(H,L_H)$ when all paths $\gamma$ in $H$ with $L_H(\gamma) = w$ terminate at the same vertex in $H$. The notion of a magic word is closely related to that of a synchronizing word. 

\begin{lemma}\label{20-12-25x} Let $(H,L_H)$ be a right-resolving, regular and follower-separated presentation of the sofic subshift $Y$. A word $w \in \mathbb W(Y)$ is synchronizing for $Y$ if and only if it is magic for $(H,L_H)$.
\end{lemma}
\begin{proof} Assume $w$ is synchronizing for $Y$, and consider two paths $p^1$ and $p^2$ in $H$ such that $L_H(p^i) = w, \ i =1,2$. Assume for a contradiction that $f_H(t_H(p^1))$ contains an element $u \in Y[0,\infty)$
 which is not in $f_H(t_H(p^2))$. Since $(H,L_H)$ is regular there is a path $z \in X_H$ such that $t_H(z_{(-\infty,-1]}) = s_H(p^2)$ and $F(L_H(z))= f_H(s_H(p^2))$. Then $L_H(z_{(-\infty,-1]})w =L_H(z_{[-\infty,-1]}p^2) \in Y(-\infty, |w|-1]$, but $L_H(z_{(-\infty,-1]})wu\notin Y$ since otherwise $wu \in F(L_H(z)) = f_H(s_H(p^2))$ and hence $u \in f_H(t_H(p^2))$ which is not the case. Hence for some large $N$, $L_H(z_{(-N,-1]})wu_{[0,N]} \notin \mathbb W(Y)$ while $L_H(z_{(-N,-1]})w \in \mathbb W(Y)$ and $wu_{[0,N]} \in \mathbb W(Y)$. This contradicts that $w$ is synchronizing and we conclude therefore that $f_H(t_H(p^1)) \subseteq f_H(t_H(p^2))$. By exchanging the roles of $p^1$ and $p^2$ in this argument we reach the conclusion that $f_H(t_H(p^1)) =f_H(t_H(p^2))$, and hence also that $t_H(p^1) = t_H(p^2)$ since $(H,L_H)$ is follower-separated. This shows that $w$ is magic for $(H,L_H)$. The converse, that magic implies synchronization, is obvious. 
 \end{proof}

Since any sofic subshift has a presentation which is right-resolving, regular and follower-separated, for example its future cover, it follows from Lemma \ref{20-12-25x} and Lemma 2.21 in \cite{Th1} that any sofic subshift has a synchronizing word:

\begin{lemma}\label{12-06-26e} Let $Y$ be a sofic subshift. Then $\sync(Y) \neq \emptyset \ \text{and} \ \eth Y \neq Y$.
\end{lemma}

\begin{lemma}\label{02-06-26}  Let $(H,L_H)$ be a right-resolving, regular and follower-separated presentation of the sofic subshift $Y$. Let $y \in Y$. The following are equivalent:
\begin{itemize}
\item[(1)] $y$ contains a synchronizing word for $Y$.
\item[(2)] $\lim_{j \to \infty} \# s_{\overline{H}}(\beta_H(y)_j) = 1$.
\item[(3)] $\lim_{j \to \infty} \# s_{\overline{H}}(\alpha_H(y)_j) = 1$.
\end{itemize}
\end{lemma}
\begin{proof} The equivalence of (2) and (3) follows because $\alpha_H(y)$ and $\beta_H(y)$ are forward asymptotic by Lemma \ref{08-10-25bxz}. (1) $\Rightarrow$ (2): If $y_{[i,j]}$ is synchronizing for $Y$ it follows from Lemma \ref{20-12-25x} that $\#s_{\overline{H}}(\beta_H(y)_k) = 1$ for $k \geq j$.

  (2) $\Rightarrow$ (1): It follows from (2) that for some $N \in \mathbb N$ there is only one vertex of $V_H$ where a path in $H$ labeled $y_{[-N,N]}$ can terminate, and then $y_{[-N,N]}$ is synchronizing by Lemma \ref{20-12-25x}.
\end{proof}

 Let $(H,L_H)$ be a right-resolving, regular and follower-separated presentation of $Y$.
Let $\eth \underline{H}$ be the subgraph of $\underline{H}$ with vertex set
$$
V_{\eth \underline{H}} := \left\{ F \in V_{\underline{H}}: \ \# F \geq 2\right\}
$$
and arrow set
$$
E_{\eth \underline{H}} := \left\{e \in E_{\underline{H}}: \ s_{\underline{H}}(e), \ t_{\underline{H}}(e) \in V_{\eth \underline{H}} \right\} .
$$
Then $\eth \underline{H}$ has no sources, but it may have sinks and we let $\eth \underline{H} = (E_{\eth \underline{H}}, V_{\eth \underline{H}})$ be the graph obtained by successively removing sinks. In other words, without changing the notation we arrange that $V_{\eth \underline{H}}$ does not contain sinks or sources.

\begin{thm}\label{11-03-26b} $\eth Y = L_{\underline{H}}(X_{\eth \underline{H}})$.
\end{thm}
\begin{proof} Let $y \in L_{\underline{H}}(X_{\eth \underline{H}})$. Then $y = L_{\underline{H}}(x)$ for some $x \in X_{\underline{H}}$ with $\# s_{\underline{H}}(x_j) \geq 2$ for all $j \in \mathbb Z$. Since $s_{\underline{H}}(x_j) \subseteq   s_{\underline{H}}(\alpha_{{H}}(y)_j)$ for all $j$, it follows from Lemma \ref{02-06-26} that $y \in \eth Y$. For the converse note that for each $y \in Y$ the number $\# s_{\underline{H}}(\alpha_{{H}}(y)_j)$ is non-increasing in $j$. Therefore, when $y \in \eth Y$ it follows from Lemma \ref{02-06-26} that $\alpha_{{H}}(y) \in \eth \underline{H}$. Hence $y = L_{\underline{H}}(\alpha_{{H}}(y)) \in  L_{\underline{H}}(X_{\eth \underline{H}})$.
\end{proof}

\begin{cor}\label{15-03-26} $\eth Y$ is a sofic subshift.
\end{cor}

We can now iterate the definition and define \emph{the $n$'th desynchronization} $\eth^n Y$ of $Y$ inductively:
$$
\eth^{n+1} Y := \eth \eth^n Y.
$$ 
This gives us a decreasing sequence 
\begin{equation}\label{12-06-26g}
Y \supseteq \eth Y \supseteq \eth^2 Y \supseteq \eth^3Y \supseteq \cdots 
\end{equation}
of sofic subshifts of $Y$. Note that $\eth^{n} Y \neq \eth^{n+1} Y$ unless $ \eth^{n} Y = \emptyset$.

\begin{rem}\label{12-06-26f} As alluded to in the introduction the definition of $\eth Y$ is similar in spirit to the definition of the derived shift $\partial Y$ in \cite{Th0}. The precise relation is simple:
$$
\partial Y = \eth R(Y),
$$
where $R(Y)$ is the non-wandering part of $Y$, i.e. the closure of the periodic orbits in $Y$. In particular, when $Y$ is irreducible the desynchronization of $Y$ is the same as the derived shift of $Y$, cf. \cite{Th1}. However, beyond the non-wandering case the two notions disagree. As a simple illustration of this consider the following labeled graph:

\begin{equation*}
 \begin{tikzpicture}[node distance={35mm}, thick, main/.style = {draw, circle}] 
\node[main] (1) {}; 
\node[main] (2) [right of=1] {} ; 
\draw[->] (1) to [out=50,in=100,looseness=10] node[above,pos=0.5] {$0$} (1);
\draw[->] (1) to [out=160,in=210,looseness=10] node[left ,pos=0.5] {$1$} (1); 
\draw[->] (1) to [out=30,in=150,looseness=1] node[above,pos=0.5] {$2$} (2);
\draw[->] (2) to [out=50,in=100,looseness=10] node[above,pos=0.5] {$1$} (2);
\draw[->] (2) to [out=30,in=330,looseness=10] node[below,pos=0.5] {$0$} (2);
\end{tikzpicture} 
\end{equation*}
If $Y$ is the sofic subshift presented by this graph we have that $\partial Y = \emptyset$ while $\eth Y$ is the full 2-shift. This difference between $\partial$ and $\eth$ has big consequences when the constructions are iterated, also for irreducible or mixing sofic subshifts. For example, for any sofic subshift $Z$ it is possible to construct a mixing sofic subshift $Y$ such that the second desynchronization $\eth^2 Y$ of $Y$ is conjugate to $Z$ while the second derived shift space $\partial^2 Y$ of $Y$ is empty. As we shall point out in Example \ref{17-07-26} below, the sofic shift presented by the labeled graph \eqref{17-07-26b} has by chance this property when $Z$ is the full two-shift.\footnote{The work on the paper \cite{Th0} was done almost 25 years ago, but I remember considering the subshift $\eth Y$ already then. It seemed to be useless, however, because I could not prove that it is always sofic.}
\end{rem}

\begin{lemma}\label{26-07-26} Let $y \in \eth Y$ and set $v = s_{\underline{H}}(\alpha_H(y)_0)$. Then $v \in V_{\eth \underline{H}}$ and \begin{equation*}
 f_{\eth \underline{H}}(v) = F_{\eth Y}(y),
 \end{equation*}
where we have used the symbol $ F_{\eth Y}(y)$ for the follower set in $\eth Y$ of $y \in \eth Y$ in order to distinguish it from $F(y)$; its follower set in $Y$.
\end{lemma}
\begin{proof} It follows from Lemma \ref{02-06-26} that $\alpha_H(y)_j \in E_{\eth \underline{H}}$ for all $j \in \mathbb Z$; in particular also that $v \in V_{\eth \underline{H}}$. Then Theorem \ref{11-03-26b} implies that $f_{\eth \underline{H}}(v) \subseteq F_{\eth Y}(y)$. On the other hand, if $z \in \eth Y[0,\infty)$ and $y_{(-\infty,-1]}z \in \eth Y$, we have that $s_{\underline{H}}(\alpha_H(y_{(-\infty,-1]}z )_0) = v$ while $\alpha_H(y_{(-\infty,-1]}z ) \in X_{\eth \underline{H}}$ by Lemma \ref{02-06-26}. Since $L_{\eth \underline{H}}(\alpha_H(y_{(-\infty,-1]}z )_{[0,\infty)}) =z$, it follows that $z \in f_{\eth \underline{H}}(v)$.
\end{proof}

\begin{lemma}\label{26-07-26a} 
\begin{itemize}
\item[(a)] $(\eth \underline{H}, L_{\eth \underline{H}})$ is right-resolving and regular, and
\item[(b)] the merged graph $([\eth \underline{H}], L_{[\eth \underline{H}]})$ of  $(\eth \underline{H}, L_{\eth \underline{H}})$ is isomorphic, as a labeled graph, to the future cover $(\mathbb K(\eth Y),L_{\mathbb K(\eth Y)})$ of $\eth Y$.
\end{itemize}
\end{lemma}
\begin{proof} (a)  It is obvious that $\eth \underline{H}$, as a subgraph of $\underline{H}$, is right-resolving. To show that $\eth \underline{H}$ is regular, let $v \in V_{\eth \underline{H}}$. Then $v = s_{\underline{H}}(\alpha_H(y')_0)$ for some $y' \in Y$. Since $\eth \underline{H}$ has no sinks there is a $z' \in X_{\eth \underline{H}}[0,\infty)$ such that $s_{\eth \underline{H}}(z') = v$. Then $z := \alpha_H(y')_{(-\infty,-1]}z' \in X_{\eth \underline{H}}$ and $y:= L_{\eth \underline{H}}(z) \in \eth Y$ by Theorem \ref{11-03-26b}. Since $v = s_{\underline{H}}(\alpha_H(y')_0) = s_{\underline{H}}(\alpha_H(y)_0)$, it follows from Lemma \ref{26-07-26} that $v$ is regular in  $(\eth \underline{H}, L_{\eth \underline{H}})$.

(b)  It follows from (a), Lemma \ref{10-09-25} and Theorem \ref{11-03-26b}, that $([\eth \underline{H}], L_{[\eth \underline{H}]})$ is a right-resolving, regular and follower-separated presentation of $\eth Y$. Hence, by Corollary 3.3 of \cite{Th1} all we need to do is to show that the number $\# V_{[\eth \underline{H}]}$ of vertices in $[\eth \underline{H}]$ is the same as the number of follower sets in $\eth Y$. In symbols we need to show that
\begin{equation}\label{03-06-26b}
\left\{f_{\eth \underline{H}}(v) : \ v \in V_{\eth \underline{H}} \right\} = \left\{ F_{\eth Y}(y) : \ y \in \eth Y \right\},
\end{equation}
where we, as in Lemma \ref{26-07-26}, use the symbol $  F_{\eth Y}(y)$ for the follower set in $\eth Y$ of $y \in \eth Y$ in order to distinguish it from $F(y)$; its follower set in $Y$. Note that \eqref{03-06-26b} follows from Lemma \ref{26-07-26} if we can show, that when $v \in V_{\eth \underline{H}}$, there is a $y \in \eth Y$ such that $v= s_{\underline{H}}(\alpha_H(y)_0)$. This was done in the proof of (a).
\end{proof}

\begin{ex}\label{17-07-26} Consider again the mixing sofic subshift $Y$ presented by the labeled graph $(G,L_G)$ depicted in \eqref{17-07-26b}. The graph is the minimal right-resolving presentation of $Y$; in particular, it is right-resolving, regular and follower-separated. The extended future cover $(\K(Y),L_{\K(Y)})$ of $Y$ agrees with the future cover, i.e. $(\mathbb K(Y),L_{\mathbb K(Y)}) = (\K(Y),L_{\K(Y)})$, and it is presented by the labeled graph \eqref{17-07-26d}. The labeled graph $(\eth \underline{G}, L_{\underline{G}})$ is here presented by the following labeled graph. 
\begin{equation*}
 \begin{tikzpicture}[node distance={35mm}, thick, main/.style = {draw, circle}] 
\node[main] (1) {$\{a,c\}$};
\node[main] (2) [right of=1] {$\{b,c\}$};
\node[main] (3) [below of=2] {$\{a,b,c\}$};
\draw[->] (3) to [out=260,in=290,looseness=10] node[below,pos=0.5] {$0$} (3); 
\draw[->] (3) to [out=10,in=350,looseness=10] node[right,pos=0.5] {$1$} (3); 
\draw[->] (3) to [out=100,in=300,looseness=0.5] node[left,pos=0.5] {$2$} (1); 
\draw[->] (1) to [out=40,in=150,looseness=1.5] node[above,pos=0.5] {$0$} (2); 
\draw[->] (1) to [out=20,in=170,looseness=1] node[above,pos=0.5] {$1$} (2); 
\draw[->] (2) to [out=190,in=350,looseness=0.5] node[below,pos=0.5] {$0$} (1); 
\draw[->] (2) to [out=220,in=320,looseness=1] node[below,pos=0.3] {$1$} (1); 
 \end{tikzpicture} 
\end{equation*}

It follows from Lemma \ref{26-07-26a} that the Krieger cover $(\mathbb K(\eth Y), L_{\mathbb K(\eth Y)})$ of the desynchronization $\eth Y$ in this case is presented by the following labeled graph.  

\begin{equation*}
 \begin{tikzpicture}[node distance={35mm}, thick, main/.style = {draw, circle}] 
\node[main] (1) {};
\node[main] (2) [below of=1] {$\{a,b,c\}$};
\draw[->] (2) to [out=260,in=290,looseness=10] node[below,pos=0.5] {$0$} (2); 
\draw[->] (2) to [out=10,in=350,looseness=10] node[right,pos=0.5] {$1$} (2);
\draw[->] (1) to [out=220,in=140,looseness=10] node[left,pos=0.5] {$0$} (1); 
\draw[->] (1) to [out=300,in=45,looseness=10] node[right,pos=0.5] {$1$} (1); 
\draw[->] (2) to [out=90,in=270,looseness=0.5] node[right,pos=0.5] {$2$} (1); 
\end{tikzpicture} 
\end{equation*}
We note, en passant, that $\eth^2 Y$ is the full 2-shift while $\partial^2 Y = \emptyset$.

\end{ex}

 We end the paper by showing that the conjugacy $\psi_{\K} : X_{\K(Y)} \to X_{\K(Z)}$ from Theorem \ref{21-01-26} takes $X_{\eth \K(Y)}$ onto $X_{\eth \K(Z)}$, and that the resulting conjugacy $\psi_{\K} : X_{\eth \K(Y)} \to X_{\eth \K(Z)}$ is compatible with the conjugacy $(\psi|_{\eth Y})_\mathbb K : \ \mathbb K(\eth Y) \to \mathbb K(\eth Z)$ obtained by applying Krieger's theorem to the conjugacy $\psi|_{\eth Y} : \eth Y \to \eth Z$. For this we return to the setting of Theorem \ref{25-03-26}. Thus $(G,L_G)$ and $(H,L_H)$ are right-resolving presentations of the sofic subshifts $Y$ and $Z$, respectively, and there are conjugacies $\phi: X_G \to X_H$ and $\psi : Y \to Z$ such that \eqref{09-10-25axy} of Lemma \ref{12-02-26bx} commutes.

\begin{lemma}\label{03-05-26a} 
Given an element $y \in Y$ and $j \in \mathbb Z$, set
$$
L_G^{-1}(y)_{[j,\infty)} := \left\{ x_{[j,\infty)} : \ x \in L_G^{-1}(y)\right\} .
$$ 
Then
$$
\lim_{j \to \infty} \# L_G^{-1}(y)_{[j,\infty)} = \lim_{j \to \infty} \# L_H^{-1}(\psi(y))_{[j,\infty)} 
$$
for all $y  \in Y$.
\end{lemma}
\begin{proof} Since $\# L_G^{-1}(y)_{[j,\infty)} \geq \# L_G^{-1}(y)_{[j+1,\infty)}$, the limit $\lim_{j \to \infty} \# L_G^{-1}(y)_{[j,\infty)}$ exists, and for a similar reason so does $\lim_{j \to \infty} \# L_H^{-1}(\psi(y))_{[j,\infty)}$. Note that $L_H^{-1}(\psi(y)) = \phi(L_G^{-1}(y))$. Let $\kappa$ be a window size for $\phi$. Then
$$
  L_H^{-1}(\psi(y))_{[j+\kappa,\infty)}  = \phi(L_G^{-1}(y))_{[j+\kappa,\infty)}  = \phi(L_G^{-1}(y)_{[j,\infty)}),
  $$
  implying that
  $$
   \# L_H^{-1}(\psi(y))_{[j+\kappa,\infty)}  \leq \# L_G^{-1}(y)_{[j,\infty)},
   $$
and hence 
  $$
 \lim_{j \to \infty} \# L_H^{-1}(\psi(y))_{[j,\infty)} \le \lim_{j \to \infty} \# L_G^{-1}(y)_{[j,\infty)} .
 $$
 By symmetry we have the stated equality.
 \end{proof}

\begin{lemma}\label{23-04-26c} Let $C$ be a source component in $\underline{H}$. Then $x = \alpha_H(L_{\underline{H}}(x)) =\beta_H(L_{\underline{H}}(x))$ for all $x \in X_C$.
\end{lemma}
\begin{proof} Note that $\alpha_H(L_{\underline{H}}(x)) \in \underline{H}$ by definitions and that $\alpha_H(L_{\underline{H}}(x))$ is forward asymptotic to $\beta_H(L_{\underline{H}}(x))$ by Lemma \ref{08-10-25bxz}. Since $x = \beta_H(L_{\underline H}(x))$ by Corollary \ref{20-01-26h} and $C$ is a source component in $\underline{H}$ it follows that $\alpha_H(L_{\underline{H}}(x)) \in X_C$. In particular, $\# s_{\underline{H}}(x_j) = \# s_{\underline{H}}(\alpha_H(L_{\underline{H}}(x))_j) = M(C)$ for all $j$. It follows first that $s_{\underline{H}}(x_j) =  s_{\underline{H}}(\alpha_H(L_{\underline{H}}(x))_j)$ for all $j$, and then that $x= \alpha_H(L_{\underline{H}}(x))$. 
\end{proof}

\begin{lemma}\label{23-04-26d} Let $C$ be a component in $\overline{H}$. The following are equivalent:
\begin{itemize}
\item[1)] There is an element $y \in Y$ such that $\alpha_H(y)$ is forward asymptotic to $C$.
\item[2)] There is an element $y \in Y$ such that $\beta_H(y)$ is forward asymptotic to $C$.
\item[3)] $C \subseteq \underline{H}$.
\end{itemize}
\end{lemma}
\begin{proof} 1) $\Leftrightarrow$ 2) by Lemma \ref{08-10-25bxz}. 1) $\Rightarrow$ 3) is clear because $\alpha_H(Y) \subseteq X_{\underline{H}}$. 3) $\Rightarrow$ 1): There is a path in $\underline{H}$ to $C$ from a source component $C'$ in $\underline{H}$ and therefore also a path $z \in X_{\underline{H}}[0,\infty)$ which is forward asymptotic to $C$ and has $s_{\underline{H}}(z) \in C'$. It follows from Lemma \ref{23-04-26c} that there is a $y \in Y$ such that $t_{\underline{H}}(\alpha_H(y)_{(-\infty,-1]}) = s_{\underline{H}}(z)$. Then $\alpha_H(y)_{(-\infty,-1]}z \in X_{\underline{H}}$ is forward asymptotic to $C$. Set $z' := L_{\underline{H}}(\alpha_H(y)_{(-\infty,-1]}z )$. Since
$$
\alpha_H(L_{\underline{H}}(\alpha_H(y))) =\alpha_H(y)
$$
it follows from Lemma \ref{11-11-24axz} that  
$$
\alpha_H(z') =\alpha_H(y)_{(-\infty,-1]}z.
$$ 
In particular, $\alpha_H(z')$ is forward asymptotic to $C$ since $z$ is. Thus 1) holds.
\end{proof}

\begin{lemma}\label{03-05-06bx} In the setting of Theorem \ref{18-01-26d}, let $C \subseteq \underline{G}$ and $C'\subseteq \underline{H}$ be components such that $\underline{\phi}(X_C) = X_{C'}$. Then $M(C') = M(C)$.
\end{lemma}
\begin{proof} It follows from Lemma \ref{23-04-26d} that we can choose $y \in Y$ such that $\alpha_G(y)$ is forward asymptotic to $C$. By Lemma \ref{08-10-25bxz} $\alpha_G(y)$ is forward asymptotic to $\beta_G(y)$ which implies that $\lim_{j \to \infty} \# L_G^{-1}(y)_{[j,\infty)} = \lim_{j \to \infty} \# s_{\underline{G}}(\beta_G(y)_j) = \lim_{j \to \infty} \# s_{\underline{G}}(\alpha_G(y)_j) = M(C)$. Similarly, $\lim_{j \to \infty} \# L_H^{-1}(\psi(y)) = \lim_{j \to \infty} \# s_{\underline{H}}(\alpha_H(\psi(y))_j)$. Since $\alpha_H(\psi(y)) = \underline{\phi}(\alpha_G(y))$ is forward asymptotic to $C'$ we conclude that $\lim_{j \to \infty} \# L_H^{-1}(\psi(y)) = 
M(C')$ and then from Lemma \ref{03-05-26a} that $M(C)=M(C')$.
\end{proof}

When we specialize to the case where $G$ and $H$ are the future covers of $Y$ and $Z$, respectively, it follows from Lemma \ref{03-05-06bx} that the set of multiplicities of components in the extended future cover of a sofic subshift is a conjugacy invariant. It also leads to the following

\begin{lemma}\label{13-06-26} Let $\psi : Y \to Z$ be a conjugacy of sofic subshifts, and $\psi_{\K} : X_{\K(Y)} \to X_{\K(Z)}$ the conjugacy from Theorem \ref{21-01-26}. Then
$$
\psi_{\K}(X_{\eth \K(Y)})  = X_{\eth \K(Z)} .
$$
\end{lemma}
\begin{proof} Let $y \in X_{\eth \K(Y)}$. Then $y$ is forward asymptotic in $X_{\K(Y)}$ to a component $C$ in $\K(Y)$ of multiplicity $M(C) \geq 2$. There is also a component $C'$ in $\K(Z)$ such that $\psi_{\K}(y)$ is forward asymptotic to $C'$. Then a forward condensation point of the orbit of $y$ in $X_{\K(Y)}$ is mapped into $X_{C'}$ which is enough to conclude that $\psi_{\K}(X_C) = X_{C'}$. It follows therefore from Lemma \ref{03-05-06bx} that $M(C') = M(C) \geq 2$. In particular, $\psi_{\K}(y) \in X_{\eth \K(Z)}$. This shows that $\psi_{\K}(X_{\eth \K(Y)}) \subseteq X_{\eth \K(Z)}$. The desired conclusion follows by symmetry.
\end{proof} 

The layers of desynchronization in a sofic subshift are conjugacy invariants: If $\psi : Y \to Z$ is a conjugacy of sofic subshifts, then
$$
\psi(\eth^n Y) = \eth^n Z
$$
for all $n$. This follows from an argument given for SFTs in Theorem 2.1.10 of \cite{LM} and repeated for general subshifts in Lemma 4.4 of \cite{Th0}. More in the spirit of this work, but much less direct, it follows also by combining Theorem \ref{11-03-26b}, Theorem \ref{21-01-26} and Lemma \ref{13-06-26}.

\begin{lemma}\label{11-06-26} In the setting of Theorem \ref{21-01-26}, let $(\psi|_{\eth Y})_\mathbb K : \mathbb K(\eth Y) \to \mathbb K(\eth Z)$ be the conjugacy from Krieger's theorem applied to $\psi|_{\eth Y} : \eth Y \to \eth Z$. Then 
$$
 (\psi|_{\eth Y})_\mathbb K \circ  f_{\eth \K(Y)}   =   f_{\eth \K(Z)} \circ \psi_{\K} 
$$
on $X_{\eth \K(Y)}$.
\end{lemma}
\begin{proof} We shall need the following observations:
\begin{itemize}
\item[a)] $f_{\eth \K(Y)} \circ \alpha_{\mathbb K(Y)} = \alpha_{\eth Y}$ on $\eth Y$, and
\item[b)] $\alpha_{\mathbb K(Y)}(\eth Y)$ is dense in $X_{\eth \K(Y)}$.
\end{itemize}

We prove first a): 
Let $y \in \eth Y$. Then $\alpha_{\mathbb K(Y)}(y) \in X_{\eth \K(Y)}$ by Lemma \ref{02-06-26} and 
\begin{align*}
& s_{\mathbb K(\eth Y)}(f_{\eth \K(Y)} \circ \alpha_{\mathbb K(Y)}(y)_j ) = f_{\eth \K(Y)}(s_{\eth \K(Y)}(\alpha_{\mathbb K(Y)}(y)_j)) \\
&= F_{\eth Y}(\sigma^j(y)) = s_{\mathbb K(\eth Y)}(\alpha_{\eth Y}(y)_j)
\end{align*}
for all $j \in \mathbb Z$, where the penultimate equality follows from Lemma \ref{26-07-26}. Since $L_{\mathbb K(\eth Y)}(f_{\eth \K(Y)} \circ \alpha_{\mathbb K(Y)}(y)) = L_{\eth \K(Y)}(\alpha_{\mathbb K(Y)}(y)) =  L_{ \K(Y}(\alpha_{\mathbb K(Y)}(y)) = y = L_{\mathbb K(\eth Y)}(\alpha_{\eth Y}(y))$, we see that a) holds.

And then b):  Let $x \in X_{\eth \K(Y)}$. Let $M \in \mathbb Z$. It follows from Lemma \ref{10-02-25dxxz} that there is a periodic point $p \in \Per(Y)$ and an element $x'\in X_{\K(Y)}$ such that $x'_j = x_j$ for $j \geq M$ and $x'$ is backward asymptotic to $\alpha_{\mathbb K(Y)}(p)$. Since $\alpha_{\mathbb K(Y)}(L_{\K(Y)}(\alpha_{\mathbb K(Y)}(p)))  = \alpha_{\mathbb K(Y)}(p)$ it follows from Lemma \ref{11-11-24axz} that $x' = \alpha_{\mathbb K(Y)}(L_{\K(Y)}(x'))$. Note that $x' \in X_{\eth \K(Y)}$ since $x'$ is forward asymptotic to $x \in X_{\eth \K(Y)}$. Hence $L_{\K(Y)}(x') \in \eth Y$ by Theorem \ref{11-03-26b} and b) follows.

Thanks to a) and b) the rest of the proof is a direct verification: Let $y \in \eth Y$. Using a), the commutative diagram (7) from Krieger's theorem, Theorem 2.11 in \cite{Th1}, and \eqref{04-08-26} from Theorem \ref{21-01-26} we find that
\begin{align*}
& (\psi|_{\eth Y})_\mathbb K \circ  f_{\eth \K(Y)} \circ \alpha_{\mathbb K(Y)}(y) = (\psi|_{\eth Y})_\mathbb K\circ \alpha_{\eth Y}(y) = \alpha_{\eth Z} \circ \psi(y)\\
& = f_{\eth \K(Z)} \circ \alpha_{\mathbb K(Z)} \circ \psi (y) =  f_{\eth \K(Z)} \circ \psi_{\K} \circ \alpha_{\mathbb K(Y)}(y).  
\end{align*}
Thanks to b) this completes the proof.
\end{proof}

The last two lemmas and Theorem \ref{21-01-26} can now be combined to give the following theorem, showing that the realization of the future cover of the desynchronization $\eth Y$ as the merged graph of $\eth \K(Y)$ is a canonical construction.

\begin{thm}\label{18-07-26g} Let $\psi : Y \to Z$ be a conjugacy of sofic subshifts. Then the diagram
\begin{equation*}
\xymatrix{ X_{\eth \underline{\mathbb K}(Y)} \ar[dd]_-{f_{\eth \underline{\mathbb K}(Y)}} \ar[rd]_-{\psi_{\underline{\mathbb K}}}  \ar@{^{(}->}[rr]   && X_{\underline{\mathbb K}(Y)} \ar'[d][dd]^-{f_{\underline{\mathbb K}(Y)}}   \ar[rd]_-{\psi_{\underline{\mathbb K}}} \\
&  X_{\eth \underline{\mathbb K}(Z)}  \ar[dd]^-{f_{\eth \underline{\mathbb K}(Z)}} \ar@{^{(}->}[rr]  && X_{\underline{\mathbb K}(Z)} \ar[dd]^-{f_{\underline{\mathbb K}(Z)}} \\
X_{\mathbb K(\eth Y)}  \ar[dd]_-{L_{\mathbb K(\eth Y)}} \ar[rd]_-{({\psi|_{\eth Y}})_{\mathbb K}} && X_{\mathbb K(Y)} \ar[dd]^-{L_{\mathbb K(Y)}}   \ar[rd]_-{\psi_\mathbb K}  \\
& X_{\mathbb K(\eth Z)} \ar[dd]^(.7){L_{\mathbb K(\eth Z)}} && X_{\mathbb K(Z)} \ar[dd]^-{L_{\mathbb K(Z)}} \\
\eth Y \ar[rd]_-{\psi|_{\eth Y}} \ar@{^{(}->}'[r][rr] & & Y \ar[rd]_-\psi\\
&  \eth Z \ar@{^{(}->}[rr]  & & Z}
\end{equation*}
commutes.
\end{thm}

\end{document}